%% file: ms.tex
\documentclass[onefignum,onetabnum]{siamart190516}

\input{SISC_header}
\newif\iflong
\longfalse
\graphicspath{ {Figures/} }

\headers{Error Estimation in Multiphysics Systems}{M. Narayanamurthi, U. R\"omer, A. Sandu}

\title{Goal-oriented a posteriori estimation of numerical errors in the solution of multiphysics systems\thanks{The work of MN and AS has been supported in part by NSF through awards NSF ACI–1709727 and NSF CCF–1613905, AFOSR through the award AFOSR DDDAS 15RT1037, and by the Computational Science Laboratory at Virginia Tech.}}

\author{Mahesh Narayanamurthi\thanks{Computational Science Laboratory, Department of Computer Science, Virginia Tech, Blacksburg, VA 24060 
  (\email{maheshnm@vt.edu}).}
\and Ulrich R\"{o}mer\thanks{Institut f\"ur Dynamik und Schwingungen, Technische Universit\"{a}t Braunschweig, Schleinitzstr. 20, 38106 Braunschweig, Germany
  (\email{u.roemer@tu-braunschweig.de}).}
\and Adrian Sandu\thanks{Computational Science Laboratory, Department of Computer Science, Virginia Tech, Blacksburg, VA 24060 
  (\email{asandu@vt.edu}).}}

\usepackage{amsopn}

\newif\iflong
\longfalse

\usepackage{pdfpages}

\begin{document}

\includepdf[landscape=false,pages=-]{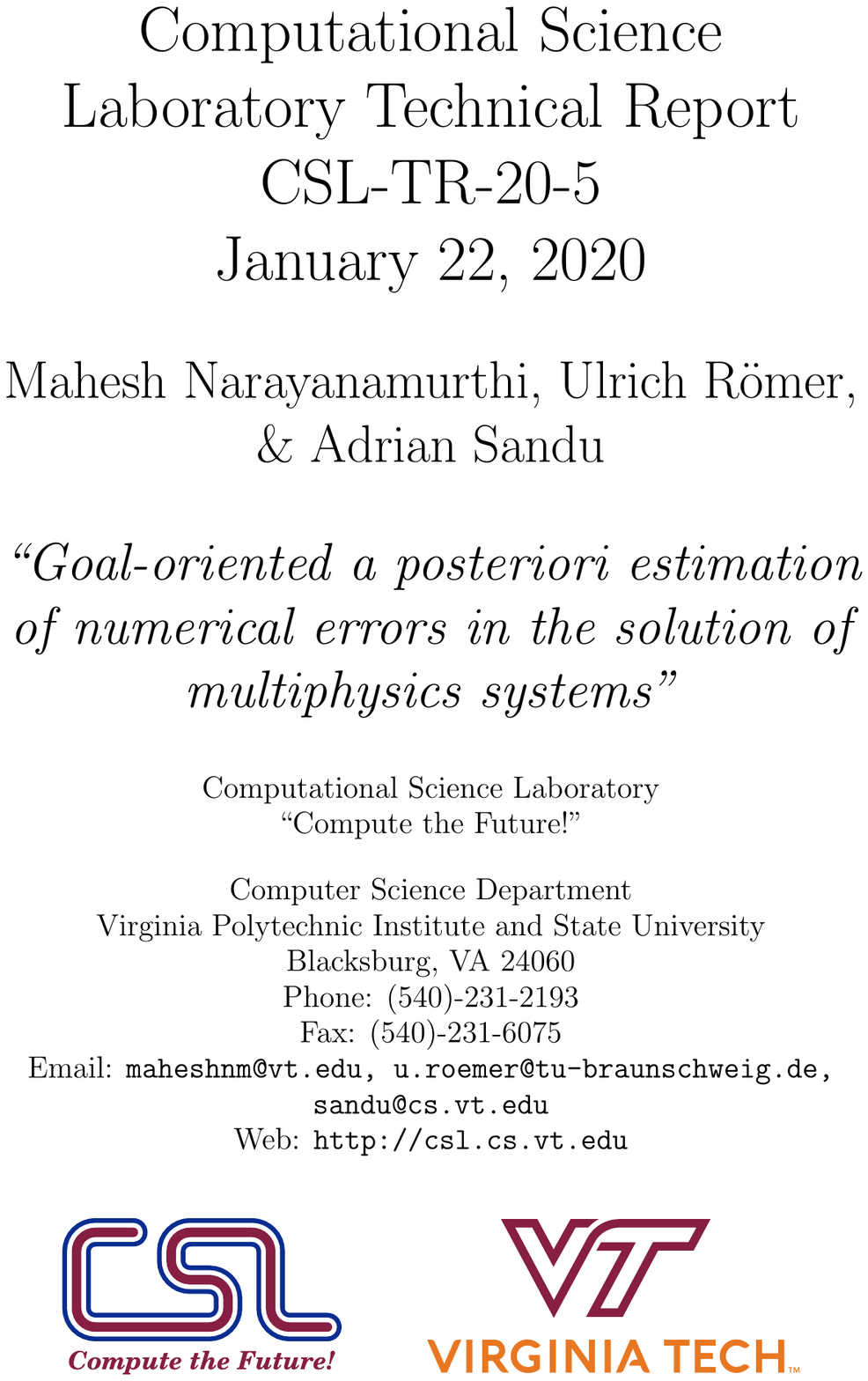}

\begin{abstract}
This paper develops a general methodology for a posteriori error estimation in time-dependent multiphysics numerical simulations. The methodology builds upon the generalized-structure additive Runge--Kutta (GARK)  approach to time integration.  GARK provides a unified formulation of multimethods that simulate complex systems by applying different discretization formulas and/or different time steps to individual components of the system. We derive discrete GARK adjoints and analyze their time accuracy. Based on the adjoint method, we establish computable a posteriori identities for the impacts of both temporal and spatial discretization errors on a given goal function. Numerical examples with reaction-diffusion systems illustrate the accuracy of the derived error measures. Local error decompositions are used to illustrate the power of this framework in adaptive refinements of both temporal and spatial meshes. 
\end{abstract}
    


\input{introduction.tex}


\input{discreteadjoints.tex}


\input{orderconvergence.tex}

\input{impactoferrors.tex}
\input{implementation.tex}

\input{numerics.tex}



\input{conclusions.tex}

\section*{Acknowledgments}
We would like to thank Steven Roberts and Arash Sarshar for their time with reviewing our implementation of these methods. We would also like to thank Stefan Wild from Argonne National Laboratory for carefully reading the manuscript and for giving valuable remarks. 

\bibliography{Bib/gark_adjoints,Bib/sandu,Bib/fenics}
\bibliographystyle{unsrt}

\newpage



\end{document}

%% file: SISC_header.tex
\usepackage{subcaption}
\usepackage{placeins}
\usepackage{pgf,tikz,pgfplots}
\pgfplotsset{compat=1.14}
\usepackage{mathrsfs}
\usetikzlibrary{arrows}

\newcommand{\ho}[1]{\bar{#1}}
\newcommand{\y}{y_\tau}
\newcommand{\adjoint}{\lambda_\tau}

\newcommand{\out}{\overline{\Psi}}
\newcommand{\outd}{\Psi}
\newcommand{\outfull}{Q}
\newcommand{\outfullvec}{q}
\newcommand{\F}{\mathrm{F}}

\newcommand{\parfrac}[2]{\frac{\partial #1}{\partial #2}}
\newcommand{\st}[1]{^{\{\mathrm{#1}\}}}
\newcommand{\stran}[1]{^{\{\mathrm{#1}\}\top}}
\newcommand{\error}{\mathcal{E}}
\newcommand{\GARK}{^{\scriptscriptstyle\mathrm{GARK}}}

\newcommand{\Id}{\mathbf{I}}
\newcommand{\Zero}{\mathbf{O}}
\renewcommand{\Re}{\mathds{R}}

\newcommand{\fun}{\mathsf{f}}        
\newcommand{\Jac}{\mathsf{J}}       
\newcommand{\Mass}{\mathsf{M}}   
\newcommand{\Advection}{\mathsf{D}}   
\newcommand{\Diffusion}{\mathsf{A}}   
\newcommand{\source}{\mathsf{s}}   
\newcommand{\Lag}{\mathcal{L}}     
\newcommand{\res}{\mathsf{r}}                 
\newcommand{\resfem}{\res^{\textsc{fem}}}     
\newcommand{\resx}{\res^{\textsc{space}}}     
\newcommand{\rest}{\res^{\textsc{time}}}     
\newcommand{\ncomp}{\mathrm{P}} 
\newcommand{\nsteps}{\mathrm{N}} 

\usepackage{dsfont}
\usepackage{siunitx}
\usepackage{mathtools}
\usepackage{lipsum}
\usepackage{amsfonts}
\usepackage{graphicx}
\usepackage{epstopdf}
\usepackage{algorithmic}
\usepackage{multirow}
\ifpdf
  \DeclareGraphicsExtensions{.eps,.pdf,.png,.jpg}
\else
  \DeclareGraphicsExtensions{.eps}
\fi

\usepackage{tabularx}
\newcolumntype{Y}{>{\centering\arraybackslash}X}

\newsiamremark{remark}{Remark}
\newsiamremark{hypothesis}{Hypothesis}
\crefname{hypothesis}{Hypothesis}{Hypotheses}
\newsiamthm{claim}{Claim}

%% file: introduction.tex
\section{Introduction}
\label{sec:introduction}
Many modern science and engineering  fields rely on complex computer simulations of multiphysics systems. Such systems are driven by multiple simultaneous physical processes, and their evolution is characterized by multiple scales in time and space. Their simulation requires specialized numerical methods to ensure computational efficiency. Multimethods use different discretization strategies and different time steps to solve individual component subsystems, and the key challenge is to ensure their accuracy and stability. Indeed, coupling effects may largely amplify discretization errors and pollute the solution. There is a substantial amount of literature on multiphysics simulations, reviewed for instance in \cite{groen2013survey}. We mention multiphysics modeling for cardiac processes \cite{quarteroni2017integrated} and parallel domain decomposition methods for coupled flow solutions \cite{cao2014parallel} as two prominent examples of them.

We are particularly interested in a posteriori error analysis for multiphysics simulations. In the context of the finite element method, this topic has been discussed in \cite{babuska1978posteriori,ainsworth2011posteriori}. Both a priori and a posteriori error analysis have been addressed in \cite{eriksson1995introduction} with emphasis towards adaptivity. From an application point of view, goal-oriented a posteriori error estimation is particularly important. Goal-oriented methods estimate the error in an output quantity of the model, called goal function in this work, which is carrying relevant physical information for a given application. Such an approach relies on duality techniques from optimal control \cite{becker2001optimal,bangerth2013adaptive,rannacher1999error}, and may lead to significantly different refined meshes when used within an adaptive algorithm, compared to residual-based approaches. The use of adjoint solutions is not restricted to a posteriori error estimation of functionals of the solution but can be employed in other situations as well, for instance, to quantify various errors in the solution of inverse or optimization problems \cite{Sandu_2015_fdvar-aposteriori,Sandu_2011_adaptive-inverse,Sandu_2011_ICCS-AdapAdj,Sandu_2011_spaceTimeDADJ,Sandu_2012_optimizationSOA}. 

A posteriori error estimation for multiphysics or related problems has already been studied in a number of works. Dual, goal-oriented operator splitting methods have been discussed in \cite{estep2008posteriori,estep2008posterioriconf,carey2009posteriori} with special emphasis on multiscale problems. These approaches have also been applied to heat transfer in \cite{estep2010posteriori}. A different goal-oriented approach to fluid-structure interaction has been reported in \cite{fick2010adjoint}. Goal-oriented a posteriori error analysis for multi-step time discretization methods have been considered  in \cite{chaudhry2015posteriori,chaudhry2016posteriori} for multi-resolution schemes in space with parallel time discretization methods. In  \cite{bengzon2010adaptive}, an adaptive method for fluid-structure interaction with duality-based error estimates was introduced. The same authors have addressed stationary multiphysics problems in \cite{larson2008adaptive}, relying again on dual solution to capture the different error contributions. 

Distinct from the approaches presented in the literature, this work relies on the discrete adjoint equations to propagate numerical errors. This allows to establish a general framework for additively-split problems, covering multiphysics, partitioned, and coupled problems. The method we present is based on the framework of generalized-structure additive Runge Kutta methods (GARK)\cite{Sandu_2015_GARK}. Multirate GARK schemes have already been successfully applied in a multiphysics context \cite{Sandu_2016_GARK-MR}. Here we develop a general framework for a posteriori error analysis in the GARK setting based on duality techniques. We first show that, under mild requirements on the coefficients, the adjoint of a GARK method is again a GARK method with different coefficients. We then establish the convergence order of the discrete GARK adjoint methods. Our main contribution is an error estimation framework for space and time discretization errors, which relies on discrete GARK adjoints. We outline how the setting covers problems with different partitions, including multiphysics and coupled PDE problems. The accuracy of the different estimators and their ability to apportion error contributions to numerical solutions of individual physical processes is illustrated through problems including reaction, diffusion and advection phenomena. We also carry out adaptive refinement of spatial and temporal meshes based on the derived error estimates.  

The paper is structured in the following way. Section \ref{sec:discreteadjoints}  introduces the abstract model problem and its GARK discretization, and derives the discrete adjoint GARK method. Section \ref{sec:orderconvergence} establishes convergence orders for the discrete GARK adjoints. Section \ref{sec:impactoferrors} describes our error estimation framework. Implementation details are given in Section \ref{sec:implementation}, whereas numerical examples are reported in Section \ref{sec:numerics}. Finally, conclusions are drawn in Section \ref{sec:conclusions}.

%% file: discreteadjoints.tex
\section{GARK numerical schemes for multiphysics systems and their discrete adjoints}
\label{sec:discreteadjoints}
In this section, before introducing the GARK method and discrete GARK adjoints, we sketch the main elements of adjoint sensitivity analysis, as this will be the cornerstone of our a posteriori error estimation framework. The last subsection also contains a specific example of a GARK scheme and its adjoint, for illustration. 

\subsection{Discrete versus continuous adjoints}

Following \cite{Sandu_2006_RKdadj} we consider the ODE 
\begin{subequations}
\label{eqn:model-system}
\begin{equation}
\label{eqn:ode}
y' = \fun(t,y), \quad y(t_0)=y_0 \in \Re^d, \quad t \in (t_0,t_\F],
\end{equation}
together with the goal function 
\begin{equation}
\label{eq:output_functional}
\out\bigl(y(t_0)\bigr) = \outfull\bigl(y(t_\F)\bigr). 
\end{equation}
\end{subequations}
Such a setting is very common, for instance, in optimal control.  Let
\begin{equation}
H(t,y,\overline{\lambda}) \coloneqq \left(\overline{\lambda}(t)\right)^\top \fun(t,y)
\end{equation}
denote the Hamiltonian function of \eqref{eqn:model-system}, where $\overline{\lambda}$ is the continuous adjoint variable. From the Hamiltonian, state and adjoint dynamics are derived as
\begin{subequations}
\label{eq:Hamiltonian}%
\begin{align}
y' &= \parfrac{H(t,y,\overline{\lambda})}{\overline{\lambda}} = \fun(t,y), \\
\label{eqn:continuous-adj}
\overline{\lambda}' &= -\parfrac{H(t,y,\overline{\lambda})}{y} = -\bigl (\Jac(t,y)\bigr )^\top\,\overline{\lambda}, \qquad \overline{\lambda}(t_F) =  \outfull_y\bigl(y(t_F)\bigr),
\end{align}
\end{subequations}
where the Jacobians of the right hand side function and goal function are denoted by:
\begin{equation}
\label{eqn:rhs-jacobian}
\Jac(t,y) \coloneqq \fun_y(t,y) \coloneqq \parfrac{\fun(t,y)}{y}, \quad
Q_y(y) \coloneqq \frac{d Q(y)}{dy},
\end{equation}
 respectively. Then, the sensitivity of the continuous goal function \eqref{eq:output_functional} with respect to (perturbations of) the continuous state is given by
\begin{equation}
\label{eqn:sensitivity_cts}
\parfrac{\out\bigl( y(t_0) \bigr)}{y(t)} = \overline{\lambda}(t).
\end{equation}
In practice we discretize the differential equation \eqref{eqn:model-system} on a (possibly non-uniform) time grid:
\[
t_0 < t_1 < \cdots  < t_{\nsteps-1} < t_\nsteps =: t_F, \quad h_n \coloneqq t_{n+1}-t_n,
\quad h \coloneqq \max_{n} h_n,
\]
to obtain numerical solutions $y_n \approx y(t_n)$, and consider the goal function \eqref{eq:output_functional} evaluated at the numerical solution:
\begin{equation}
\label{eq:output_functional_discrete}
\outd\bigl(y_0\bigr) = \outfull\bigl(y_N\bigr). 
\end{equation}
Discrete adjoint variables $\lambda_n$ give the sensitivities of the numerical solution functional \eqref{eq:output_functional_discrete} with respect to perturbations in the numerical solution:
\begin{equation}
\label{eqn:sensitivity_dis}
\parfrac{\outd\bigl( y_0 \bigr)}{y_n} = \lambda_n.
\end{equation}

\subsection{Multiphysics systems and GARK discretizations}
%
As an abstract model for multiphysics problems we consider the additively-split system:
\begin{equation}
\label{eqn:split-ode}
y' = \fun(y) = \sum_{q=1}^\ncomp \fun\st{q}\left(y\right), \quad t_0 < t \le t_F, \quad y(t_0)=y_0 \in \mathbb{R}^d.
\end{equation}
Following \cite{Sandu_2015_GARK} one step of a GARK discretization method applied to \eqref{eqn:split-ode} reads:
\begin{subequations}
\label{eq:GARK_scheme}
\begin{align}
\label{eq:GARK_step}
y_{n+1} &= y_n + h_n \sum_{q=1}^\ncomp \sum_{i=1}^{s\st{q}} b_i\st{q}\, \fun\st{q}\big(T_{n,i},Y_{n,i}\st{q}\big), \\
\label{eq:GARK_stages}
Y_{n,i}\st{q} &= y_n + h_n \sum_{m=1}^\ncomp \sum_{j=1}^{s\st{m}} a_{i,j}\st{q,m}\, \fun\st{m}\big(T_{n,j},Y_{n,j}\st{m}\big),\\
T_{n,i} &= t_{n} + c_i\,h_n, \qquad q=1,\ldots,\ncomp, \quad  i=1,\ldots,s\st{q}.
\nonumber
\end{align}
\end{subequations}
For brevity, in the remainder of this section we denote:
\begin{equation}
\label{eqn:shorthand-fj}
\fun\st{q}_{n,i} \coloneqq \fun\st{q}\big(T_{n,i},Y_{n,i}\st{q}\big), \quad \Jac\st{q}_{n,i} \coloneqq \parfrac{\fun\st{q}}{y}\big(T_{n,i},Y_{n,i}\st{q}\big).
\end{equation}
The GARK scheme \eqref{eq:GARK_scheme} can be written in the following equivalent formulation using stage slopes:
\begin{subequations}
\label{eq:GARK-alt_scheme}
\begin{align}
\label{eq:GARK-alt_step}
y_{n+1} &= y_n + h_n \sum_{q=1}^\ncomp \sum_{i=1}^{s\st{q}} b_i\st{q}\, k_{n,i}\st{q}, \\
\label{eq:GARK-alt_stages}
k_{n,i}\st{q} &= \fun\st{q}\Bigl(T_{n,i},y_n + h_n \sum_{m=1}^\ncomp \sum_{j=1}^{s\st{m}} a_{i,j}\st{q,m}\, k_{n,j}\st{m}\Bigr)
,\\
\nonumber
&\qquad q=1,\ldots,\ncomp, \quad  i=1,\ldots,s\st{q}.
\end{align}
\end{subequations}

\subsection{Discrete adjoint GARK schemes}

Similar to \eqref{eqn:continuous-adj}, the continuous adjoint equation of the multiphysics system \eqref{eqn:split-ode} reads:
\begin{equation}
\label{eqn:split-continuous-adj}
\overline{\lambda}' = -\sum_{q=1}^\ncomp \bigl( \Jac\st{q}(t,y)\bigr)^\top\,\overline{\lambda}, \qquad \overline{\lambda}(t_F) =  \bigl(\outfull_y\bigl(y(t_F)\bigr) \bigr)^\top.
\end{equation}

We derive the discrete GARK adjoint in the following Lemma.
\begin{lemma}[Discrete adjoint of the GARK scheme]
\label{lem:costate_GARK_original}
The discrete adjoint of the GARK scheme \eqref{eq:GARK_scheme} with the goal function \eqref{eq:output_functional} reads:
\begin{subequations}
\begin{align}
\label{lem:costate_GARK_original_final}
\lambda_N &= Q_y^\top \big|_{y=y_N}, \\
\label{lem:costate_GARK_original_sol}
\lambda_n &= \lambda_{n+1} + \sum_{q=1}^\ncomp \sum_{i=1}^{s\st{q}} \theta_{n,i}\st{q}, \quad n=N-1, \dots, 0,\\
\label{lem:costate_GARK_original_stage}
\theta_{n,i}\st{q} &= h_n \Jac\stran{q}_{n,i} \left( b_i\st{q}\, \lambda_{n+1}  + \sum_{m=1}^\ncomp \sum_{j=1}^{s\st{m}}  a_{j,i}\st{m,q} \, \theta_{n,j}\st{m} \right), \quad   i=s\st{q},\ldots,1, \\
& \qquad q=\ncomp,\dots,1, \quad n=N-1, \dots, 0. \nonumber
\end{align}
\label{eq:costate_GARK_original}
\end{subequations}

\begin{proof}
Consider the discrete Lagrangian associated with the goal function \eqref{eq:output_functional}, evaluated at the numerical solution $y_N$, and the constraints posed by the GARK evolution equations \eqref{eq:GARK_scheme}:
\begin{multline}
\label{eq:lagrangian_GARK}
\Lag =   \outfull(y_\nsteps) + \sum_{n=0}^{\nsteps-1} \lambda_{n+1}^\top \left(y_{n+1}-y_n - h_n \sum_{r=1}^\ncomp \sum_{i=1}^{s\st{r}} b_i\st{r} \,\fun\st{r}_{n,i} \right)  \\  +  \sum_{n=0}^{\nsteps-1}  \sum_{r=1}^\ncomp\sum_{k=1}^{s\st{r}} \theta\stran{r}_{n,k} \left (  Y_{n,k}\st{r} - y_n - h_n \sum_{m=1}^\ncomp \sum_{j=1}^{s\st{m}} a_{k,j}\st{r,m}\, \fun\st{m}_{n,j} \right),
\end{multline}
where we explicitly introduced Lagrange multipliers for each of the stage equations. 
We compute
\begin{equation}
\parfrac{\Lag}{y_n} = 
\begin{cases}
Q_y(y_N) - \lambda_N^\top, & n = N, \\
\lambda_n^\top - \lambda_{n+1}^\top - \sum_{q=1}^\ncomp \sum_{i=1}^{s\st{q}} \theta\stran{q}_{n,i}, & n < N,
\end{cases}
\end{equation}
and setting it to zero leads to \eqref{lem:costate_GARK_original_final} and \eqref{lem:costate_GARK_original_sol}, respectively. Similarly, equating the derivative of the Lagrangian \eqref{eq:lagrangian_GARK} with respect to $Y_{n,i}\st{q}$ to zero yields \eqref{lem:costate_GARK_original_stage}.
\end{proof}
\end{lemma}

\begin{remark}[Interpretation of adjoint variables]
For equality-constrained optimization the values of the Lagrange multipliers at a stationary point of the Lagrangian are the sensitivities of the goal function with respect to small perturbations in the corresponding constraints \cite{Nocedal_2006_book}. In our formulation \eqref{eq:lagrangian_GARK} the multiplier $\lambda_{n+1}$ is associated with the solution equation for $y_{n+1}$. Replacing the solution equation at the $n$-th step only with a slightly perturbed equation:
\[
\widehat{y}_{n+1} = y_n + h_n \sum_{r=1}^\ncomp \sum_{i=1}^{s\st{r}} b_i\st{r} \,\fun\st{r}_{n,i} + \Delta\,y_{n+1}
\]
leads to a slightly different final solution $\widehat{y}_{N}$ and a slightly different value of the goal function $Q(\widehat{y}_{N})$.
According to the sensitivity interpretation of Lagrange multipliers we have that
\[
\lambda_{n+1}^\top = \frac{d Q(y_{N})}{d y_{n+1}}, \qquad Q(\widehat{y}_{N}) - Q(y_{N}) \approx \lambda_{n+1}^\top\, \Delta\,y_{n+1}.
\]
Similarly, the adjoint variables $\theta_{n,k}\st{r}$ are associated with a stage equation, and therefore
\[ 
\theta\stran{\textsc{r}}_{n,k} = \frac{d Q(y_{N})}{dY_{n,k}\st{r}}.
\]
\end{remark}

\begin{remark}[Constraint qualification]
In an optimization context the constraints are posed by the scheme equations \eqref{eq:GARK_scheme}. Note that the constraints \eqref{eq:GARK_step}  involve a different set of variables $(y_n,y_{n+1})$ at each step $n$. If the method is irreducible none of the stage vector calculations \eqref{eq:GARK_stages} can be avoided while preserving the result. Consequently, the linear independence constraint qualification conditions are automatically satisfied for irreducible GARK methods \eqref{eq:GARK_scheme}.
\end{remark}

\begin{remark}[Alternative formulation of GARK discrete adjoints]
\label{com:alternative-GARK-adjoint}
\iflong
The corresponding Lagrangian is:
\begin{multline}
\label{eq:lagrangian_GARK-alt}
\Lag =   \outfull(y_\nsteps) + \sum_{n=0}^{\nsteps-1} \lambda_{n+1}^\top \left(y_{n+1}-y_n - h_n \sum_{r=1}^\ncomp \sum_{i=1}^{s\st{r}} b_i\st{r} k\st{r}_{n,i} \right)  \\  +  \sum_{n=0}^{\nsteps-1}  \sum_{r=1}^\ncomp\sum_{k=1}^{s\st{r}} \left(\mu_{n,k}\st{r}\right)^{\top} \left (  k_{n,k}\st{r} - \fun\st{r}\big(y_n + h_n \sum_{m=1}^\ncomp \sum_{j=1}^{s\st{m}} a_{k,j}\st{r,m} k_{n,j}\st{m}\big) \right),
\end{multline}
with the variation:
\begin{multline}
\label{eq:delta-lagrangian_GARK-alt}
\delta\Lag =   \outfull(y_\nsteps) + \sum_{n=0}^{\nsteps-1} \lambda_{n+1}^\top \left(\delta y_{n+1}-\delta y_n - h_n \sum_{r=1}^\ncomp \sum_{i=1}^{s\st{r}} b_i\st{r} \delta k\st{r}_{n,i} \right)  \\  
+  \sum_{n=0}^{\nsteps-1}  \sum_{r=1}^\ncomp\sum_{k=1}^{s\st{r}} \left(\mu_{n,k}\st{r}\right)^{\top} \left ( \delta k_{n,k}\st{r} - \Jac\st{r} \cdot \big(\delta y_n + h_n \sum_{m=1}^\ncomp \sum_{j=1}^{s\st{m}} a_{k,j}\st{r,m} \delta k_{n,j}\st{m}\big) \right).
\end{multline}
\fi
Similar to \eqref{eqn:shorthand-fj} let
\[
Y_{n,i}\st{q} = y_n + h_n \sum_{m=1}^\ncomp \sum_{j=1}^{s\st{m}} a_{i,j}\st{q,m} \,k_{n,j}\st{m}, \quad 
\Jac\st{q}_{n,i} = \parfrac{\fun\st{q}}{y}\big(T_{n,i},Y_{n,i}\st{q}\big).
\]
It can be seen that the modified discrete adjoint equations of Lemma~\eqref{lem:costate_GARK_original} corresponding to the stage formulation \eqref{eq:GARK-alt_scheme} read:
\begin{subequations}
\label{eq:costate_GARK_K}
\begin{align}
\label{lem:costate_GARK_K_final}
\lambda_N &= \bigl(Q_y(y_N) \bigr)^\top, \\
\label{lem:costate_GARK_K_sol}
\lambda_n &= \lambda_{n+1} + \sum_{q=1}^\ncomp \sum_{i=1}^{s\st{q}} \Jac\stran{q}_{n,i}\,\mu_{n,i}\st{q},
\quad n=N-1, \dots, 0,\\
\label{lem:costate_GARK_K_stage}
\mu_{n,i}\st{q} &= h_n\, b_i\st{q}\, \lambda_{n+1} + h_n \sum_{m=1}^\ncomp \sum_{j=1}^{s\st{q}}  a_{j,i}\st{m,q}\, \Jac\stran{m}_{n,j} \,\mu_{n,j}\st{m}, \quad i = s\st{q},\dots,1, \\
&\qquad q=1,\dots,P, \quad n=N-1, \dots, 0. \nonumber
\end{align}
\end{subequations}
We conclude that 
\begin{equation}
    \theta_{n,i}\st{q} \equiv \Jac\stran{q}_{n,i}\, \mu_{n,i}\st{q}.
\end{equation}
\end{remark}

\begin{remark}[Interpretation of the alternative adjoint variables]
Using the sensitivity interpretation of adjoint variables \cite{Nocedal_2006_book}, and the fact that the Lagrange multipliers $\mu_{n,j}\st{m}$ \eqref{lem:costate_GARK_K_stage} are associated with the stage slope equations \eqref{eq:GARK-alt_stages}, we have that
\begin{equation}
\label{eqn:mu-as-sensitivity}
\mu\stran{m}_{n,j}  = \frac{d Q(y_{N})}{d k_{n,j}\st{m}}.
\end{equation}
\end{remark}

As in the case of standard Runge-Kutta methods \cite{Sandu_2006_RKdadj}, if the weights are nonzero, we can reformulate the discrete GARK adjoint as another GARK method.

\begin{lemma}
If all $b_i\st{q}\neq 0$ then the following holds:
\begin{subequations}
\label{eq:GARK_costate}
\begin{align}
\label{eq:costate_step_GARK} 
\lambda_{n} &= \lambda_{n+1} + h_n \sum_{q=1}^\ncomp \sum_{i=1}^{s\st{q}} \bar{b}_i\st{q} \, \ell_{n,i}\st{q}, \\
\label{eq:costate_stages_GARK} 
\Lambda_{n,i}\st{q} &= \lambda_{n+1} + h_n \sum_{m=1}^\ncomp \sum_{j=1}^{s\st{m}} \bar{a}_{i,j}\st{q,m} \, \ell_j\st{m}, 
\quad i = s\st{q}, \dots, 1, \\
\label{eq:costate_stages_ell} 
\ell_{n,i}\st{q} &= \Jac\stran{q}_{n,i} \, \Lambda_{n,i}\st{q},
\end{align}
with 
\begin{equation}
\label{eq:GARK_costate_coefficients}
\bar{b}_i\st{q}=b_i\st{q}, \qquad \bar{a}_{i,j}\st{m,q}=\frac{b_j\st{q}\,a_{j,i}\st{q,m}}{b_i\st{m}}.
\end{equation}
\end{subequations}
%

\begin{proof}
Defining
\begin{equation}
\Lambda_{n,i}\st{q} \coloneqq \lambda_{n+1}  +  \sum_{m=1}^\ncomp\sum_{j=1}^{s\st{m}}  \frac{a_{j,i}\st{m,q}}{b_i\st{q}}\, \theta_{n,j}\st{m},\quad
\ell_{n,i}\st{q} \coloneqq \Jac\stran{q}_{n,i}\, \Lambda_{n,i}\st{q},
\end{equation}
and replacing it in \eqref{lem:costate_GARK_original_stage} leads to:
\begin{equation}
\theta_{n,i}\st{q} = h_n\,b_i\st{q} \, \underbrace{\Jac\stran{q}_{n,i}\,\Lambda_{n,i}\st{q}}_{\eqqcolon \ell\st{q}_{n,i}},
\end{equation}
which establishes \eqref{eq:costate_step_GARK}. Moreover:
\begin{equation}
\Lambda_{n,i}\st{q} \coloneqq \lambda_{n+1}  +  h_n \sum_{m=1}^\ncomp\sum_{j=1}^{s\st{m}}  \frac{a_{j,i}\st{m,q}\,b_j\st{m}}{b_i\st{q}}\, \ell\st{m}_{n,j},
\end{equation}
and \eqref{eq:costate_stages_GARK} follows.
\end{proof}
\end{lemma}

\subsection{An example of a discrete adjoint GARK scheme}
\label{subsec:discrete_adjoint_GARK_example}
%
Consider the following two-stage IMEX GARK method \cite[Example 9]{Sandu_2015_GARK}. It is internally consistent, second order, and stiffly accurate, with one free parameter $\alpha$ and $\gamma = 1 \pm \sqrt{2}/{2}$. The choice $\alpha=\gamma$ leads to the same weights for both schemes. The Butcher tableau reads:
\begin{equation}
\label{eqn:GARK-butcher-imex22}
\begin{array}{c | cc | cc}
c_1\st{\textsc{e}} &  0 & 0 & 0 &0   \\
c_2\st{\textsc{e}} & a_{2,1}\st{\textsc{e},\textsc{e}} & 0 & a_{2,1}\st{\textsc{e},\textsc{i}} & 0 \\
\hline
c_1\st{\textsc{i}} & a_{1,1}\st{\textsc{i},\textsc{e}} & 0 & a_{1,1}\st{\textsc{i},\textsc{i}}  & 0 \\
c_2\st{\textsc{i}} &  a_{2,1}\st{\textsc{i},\textsc{e}} & a_{2,2}\st{\textsc{i},\textsc{e}}  & a_{2,1}\st{\textsc{i},\textsc{i}} & a_{2,2}\st{\textsc{i},\textsc{i}}  \\
\hline
1 & b_{1}\st{\textsc{e}} & b_{2}\st{\textsc{e}}  & b_{1}\st{\textsc{i}} & b_{2}\st{\textsc{i}}   \\
\end{array}
~\equiv~
\begin{array}{c | cc | cc}
0 &  0 & 0 & 0 &0   \\
1/(2 \alpha) & 1/(2 \alpha) & 0 & 1/(2 \alpha) & 0 \\
\hline
\gamma & \gamma & 0 & \gamma  & 0 \\
1 &  1-\alpha & \alpha  & 1-\gamma & \gamma  \\
\hline
1 & 1-\alpha & \alpha  & 1-\gamma & \gamma   \\
\end{array}.
\end{equation}
The numerical solution proceeds by computing the first explicit stage, the first implicit stage, the second explicit stage, and the second implicit stage, in this order:
\begin{equation}
\label{eqn:GARK-imex22}
\begin{split}
Y_{n,1}\st{\textsc{e}} &= y_{n},  \\
Y_{n,1}\st{\textsc{i}} &= y_{n} + h \, a_{1,1}\st{\textsc{i},\textsc{e}} \, \fun\st{\textsc{e}}_{n,1}  + h \,  a_{1,1}\st{\textsc{i},\textsc{i}}  \, \fun\st{\textsc{i}}_{n,1},  \\
Y_{n,2}\st{\textsc{e}} &= y_{n} + h \, a_{2,1}\st{\textsc{e},\textsc{e}} \, \fun\st{\textsc{e}}_{n,1}  + h \, a_{2,1}\st{\textsc{e},\textsc{i}} \, \fun\st{\textsc{i}}_{n,1}, \\
Y_{n,2}\st{\textsc{i}} &= y_{n} + h \, a_{2,1}\st{\textsc{i},\textsc{e}} \, \fun\st{\textsc{e}}_{n,1} + h \, a_{2,1}\st{\textsc{i},\textsc{i}} \, \fun\st{\textsc{i}}_{n,1} + h \, a_{2,2}\st{\textsc{i},\textsc{e}} \, \fun\st{\textsc{e}}_{n,2} + h \, a_{2,2}\st{\textsc{i},\textsc{i}} \, \fun\st{\textsc{i}}_{n,2}, \\
y_{n+1} &= y_{n} + h \, b_{1}\st{\textsc{e}} \, \fun\st{\textsc{e}}_{n,1} + h \,  b_{1}\st{\textsc{i}} \, \fun\st{\textsc{i}}_{n,1}  + h \,  b_{2}\st{\textsc{e}} \, \fun\st{\textsc{e}}_{n,2} + h \,  b_{2}\st{\textsc{i}} \, \, \fun\st{\textsc{i}}_{n,2},
\end{split}
\end{equation}
 where $\fun\st{\textsc{e}}_{n,i} \coloneqq \fun\st{\textsc{e}}(Y_{n,i}\st{\textsc{e}})$ and $\fun\st{\textsc{i}}_{n,i} \coloneqq \fun\st{\textsc{i}}(Y_{n,i}\st{\textsc{i}})$.
We note that $y_{n+1}=Y_2\st{\textsc{i}}$, which is precisely the GARK-stiff-accuracy property built into the method.


We first compute the discrete adjoint directly, via variational calculus. The first variation of the scheme  \eqref{eqn:GARK-imex22} reads
\begin{equation*}
\label{eqn:GARK-imex22-tlm}
\begin{split}
\frac{d y_{n+1}}{d y_n}
&= \Id_d +
\begin{bmatrix} h \, b_{1}\st{\textsc{e}} \Jac\st{\textsc{e}}_{n,1} & h \,  b_{1}\st{\textsc{i}} \Jac\st{\textsc{i}}_{n,1} & h \,  b_{2}\st{\textsc{e}} \Jac\st{\textsc{e}}_{n,2} & h \,  b_{2}\st{\textsc{i}} \, \Jac\st{\textsc{i}}_{n,2} \end{bmatrix}^\top \cdot \\
& \ \cdot \begin{bmatrix}
\Id_d &&& \\
-h \, a_{1,1}\st{\textsc{i},\textsc{e}} \Jac\st{\textsc{e}}_{n,1}  & \Id_d -h \,  a_{1,1}\st{\textsc{i},\textsc{i}}  \Jac\st{\textsc{i}}_{n,1} & &  \\
- h \, a_{2,1}\st{\textsc{e},\textsc{e}} \Jac\st{\textsc{e}}_{n,1}  & -h \, a_{2,1}\st{\textsc{e},\textsc{i}} \Jac\st{\textsc{i}}_{n,1} &\Id_d  & \\
- h \, a_{2,1}\st{\textsc{i},\textsc{e}} \Jac\st{\textsc{e}}_{n,1} & -h \, a_{2,1}\st{\textsc{i},\textsc{i}} \Jac\st{\textsc{i}}_{n,1} & -h \, a_{2,2}\st{\textsc{i},\textsc{e}} \Jac\st{\textsc{e}}_{n,2} & \Id_d -h \, a_{2,2}\st{\textsc{i},\textsc{i}} \Jac\st{\textsc{i}}_{n,2}
\end{bmatrix}^{-1}\, \mathbf{1}_{4d},
\end{split}
\end{equation*}
where $\Id_d$ denotes the $\mathbb{R}^{d\times d}$ identity matrix and $\mathbf{1}_{4d} \in \mathbb{R}^{sd}$ is a vector of ones. 
The discrete adjoint equation becomes:
\begin{subequations}
\label{eqn:GARK-imex22-adj}
\begin{equation}
\label{eqn:GARK-imex22-adj-sol}
\lambda_{n} = \left(\frac{d y_{n+1}}{d y_n}\right)^\top\, \lambda_{n+1} = \lambda_{n+1} + \mathbf{1}_{4d}^\top \, \boldsymbol{\theta} = \lambda_{n+1} + \theta_{n,1}\st{\textsc{e}} + \theta_{n,1}\st{\textsc{i}} + \theta_{n,2}\st{\textsc{e}} + \theta_{n,2}\st{\textsc{i}},
\end{equation}
where
\begin{align*}
\begin{split}
&\boldsymbol{\theta}=
\begin{bmatrix} \theta_{n,1}\st{\textsc{e}} \\  \theta_{n,1}\st{\textsc{i}} \\ \theta_{n,2}\st{\textsc{e}} \\  \theta_{n,2}\st{\textsc{i}} \end{bmatrix}  =\\
& 
\!\! \begin{bmatrix}
\Id_d & -h \, a_{1,1}\st{\textsc{i},\textsc{e}} \Jac\stran{\textsc{e}}_{n,1} &- h \, a_{2,1}\st{\textsc{e},\textsc{e}} \Jac\stran{\textsc{e}}_{n,1}& - h \, a_{2,1}\st{\textsc{i},\textsc{e}} \Jac\stran{\textsc{e}}_{n,1}\\
 & \Id_d -h \,  a_{1,1}\st{\textsc{i},\textsc{i}}  \Jac\stran{\textsc{i}}_{n,1} & -h \, a_{2,1}\st{\textsc{e},\textsc{i}} \Jac\stran{\textsc{i}}_{n,1} & -h \, a_{2,1}\st{\textsc{i},\textsc{i}} \Jac\stran{\textsc{i}}_{n,1} \\
  & &\Id_d  & -h \, a_{2,2}\st{\textsc{i},\textsc{e}} \Jac\stran{\textsc{e}}_{n,2} \\
 &  &  & \Id_d -h \, a_{2,2}\st{\textsc{i},\textsc{i}} \Jac\stran{\textsc{i}}_{n,2}
\end{bmatrix}^{-1}\,\!\!
\begin{bmatrix} h \, b_{1}\st{\textsc{e}} \Jac\stran{\textsc{e}}_{n,1}\,\lambda_{n+1} \\ h \,  b_{1}\st{\textsc{i}} \Jac\stran{\textsc{i}}_{n,1}\,\lambda_{n+1} \\ h \,  b_{2}\st{\textsc{e}} \Jac\stran{\textsc{e}}_{n,2}\,\lambda_{n+1} \\ h \,  b_{2}\st{\textsc{i}} \, \Jac\stran{\textsc{i}}_{n,2}\,\lambda_{n+1} \end{bmatrix}. 
\end{split}
\end{align*}
We solve the upper triangular system by backward substitution as follows:
\begin{equation}
\label{eqn:GARK-imex22-adj-stage}
\begin{split}
& \bigl( \Id_d -h \, a_{2,2}\st{\textsc{i},\textsc{i}} \Jac\stran{\textsc{i}}_{n,2} \bigr)\,\theta_{n,2}\st{\textsc{i}} \! =\! h \,  \Jac\stran{\textsc{i}}_{n,2} \!\!\left( b_{2}\st{\textsc{i}} \,\lambda_{n+1} \right), \\
& \theta_{n,2}\st{\textsc{e}} \! =\! h \,\Jac\stran{\textsc{e}}_{n,2}\!\!\left(  b_{2}\st{\textsc{e}} \lambda_{n+1} + a_{2,2}\st{\textsc{i},\textsc{e}} \,\theta_{n,2}\st{\textsc{i}} \right), \\
& \bigl( \Id_d -h \, a_{1,1}\st{\textsc{i},\textsc{i}} \Jac\stran{\textsc{i}}_{n,1} \bigr)\,\theta_{n,1}\st{\textsc{i}}\! =\! h \, \Jac\stran{\textsc{i}}_{n,1} \!\! \left(  b_{1}\st{\textsc{i}}\lambda_{n+1} +  a_{2,1}\st{\textsc{e},\textsc{i}} \,\theta_{n,2}\st{\textsc{e}}  + a_{2,1}\st{\textsc{i},\textsc{i}} \,\theta_{n,2}\st{\textsc{i}} \right), \\
& \theta_{n,1}\st{\textsc{e}}\! =\! h \, \Jac\stran{\textsc{e}}_{n,1} \!\!\left( b_{1}\st{\textsc{e}} \lambda_{n+1}
\!+\! a_{1,1}\st{\textsc{i},\textsc{e}}\,\theta_{n,1}\st{\textsc{i}} \!+\! a_{2,1}\st{\textsc{e},\textsc{e}}\,\theta_{n,2}\st{\textsc{e}}
\!+\!a_{2,1}\st{\textsc{i},\textsc{e}} \,\theta_{n,2}\st{\textsc{i}} \right).
\end{split}
\end{equation}
\end{subequations}
Equations \eqref{eqn:GARK-imex22-adj} represent the particular form the discrete adjoint equations \eqref{eq:costate_GARK_original} take for our method \eqref{eqn:GARK-imex22}. We note that the order in which the adjoint stages are evaluated in \eqref{eqn:GARK-imex22-adj-stage} is exactly the reverse of the order of evaluation of forward stages \eqref{eqn:GARK-imex22}. Moreover, implicit adjoint stages in \eqref{eqn:GARK-imex22-adj} correspond to implicit forward stages in \eqref{eqn:GARK-imex22}, and the same holds for explicit stages. Forward stages $Y_{n,i}$ need to be computed from \eqref{eqn:GARK-imex22} and stored, as the discrete adjoint formulation \eqref{eqn:GARK-imex22-adj} includes the Jacobians evaluated at these stage values.

Assuming all the weights are nonzero we scale each discrete adjoint stage by the corresponding $b$ coefficient. 
\iflong
The system \eqref{eqn:GARK-imex22-adj-sol}--\eqref{eqn:GARK-imex22-adj-stage} becomes:
\begin{equation*}
\begin{split}
\bigl( \Id_d -h \, a_{2,2}\st{\textsc{i},\textsc{i}} \Jac\st{\textsc{i}}_{n,2},\!^\top \bigr)\,\frac{\theta_{n,2}\st{\textsc{i}}}{b_{2}\st{\textsc{i}}} &= h \,  \Jac\st{\textsc{i}}_{n,2},\!^\top\, \,\lambda_{n+1}, \\
\frac{\theta_{n,2}\st{\textsc{e}}}{ b_{2}\st{\textsc{e}} } &= h \,\Jac\st{\textsc{e}}_{n,2},\!^\top\,\left( \lambda_{n+1} + \frac{b_{2}\st{\textsc{i}}\,a_{2,2}\st{\textsc{i},\textsc{e}}}{ b_{2}\st{\textsc{e}} } \,\frac{\theta_{n,2}\st{\textsc{i}}}{b_{2}\st{\textsc{i}}} \right), \\
\bigl( \Id_d -h \, a_{1,1}\st{\textsc{i},\textsc{i}} \Jac\st{\textsc{i}}_{n,1}\,\!^\top \bigr)\,\frac{\theta_{n,1}\st{\textsc{i}}}{b_{1}\st{\textsc{i}}} &= h \, \Jac\st{\textsc{i}}_{n,1}\,\!^\top\, \left(  \lambda_{n+1} +  \frac{b_{2}\st{\textsc{e}}\,a_{2,1}\st{\textsc{e},\textsc{i}}}{b_{1}\st{\textsc{i}}} \,\frac{\theta_{n,2}\st{\textsc{e}}}{b_{2}\st{\textsc{e}}}  + \frac{b_{2}\st{\textsc{i}}\,a_{2,1}\st{\textsc{i},\textsc{i}}}{b_{1}\st{\textsc{i}}} \,\frac{\theta_{n,2}\st{\textsc{i}}}{b_{2}\st{\textsc{i}}} \right), \\
\frac{\theta_{n,1}\st{\textsc{e}}}{b_{1}\st{\textsc{e}}} &= h \, \Jac\st{\textsc{e}}_{n,1}\,\!^\top\, \left( \lambda_{n+1}
+ \frac{b_{1}\st{\textsc{i}}\,a_{1,1}\st{\textsc{i},\textsc{e}}}{b_{1}\st{\textsc{e}}}\,\frac{\theta_{n,1}\st{\textsc{i}}}{b_{1}\st{\textsc{i}}} + \right .\\
\hspace*{9em} & +\left .\frac{b_{2}\st{\textsc{e}}\,a_{2,1}\st{\textsc{e},\textsc{e}}}{b_{1}\st{\textsc{e}}}\,\frac{\theta_{n,2}\st{\textsc{e}}}{b_{2}\st{\textsc{e}}}
+\frac{b_{2}\st{\textsc{i}}\,a_{2,1}\st{\textsc{i},\textsc{e}}}{b_{1}\st{\textsc{e}}} \,\frac{\theta_{n,2}\st{\textsc{i}}}{b_{2}\st{\textsc{i}}} \right), \\
\lambda_{n} & = \lambda_{n+1} + b_{1}\st{\textsc{e}}\,\frac{\theta_{n,1}\st{\textsc{e}}}{b_{1}\st{\textsc{e}}} + b_{1}\st{\textsc{i}}\,\frac{\theta_{n,1}\st{\textsc{i}}}{b_{1}\st{\textsc{i}}} + b_{2}\st{\textsc{e}}\,\frac{\theta_{n,2}\st{\textsc{e}}}{b_{2}\st{\textsc{e}}} + b_{2}\st{\textsc{i}}\,\frac{\theta_{n,2}\st{\textsc{i}}}{b_{2}\st{\textsc{i}}}.
\end{split}
\end{equation*}
\fi
Using the notation \eqref{eq:GARK_costate_coefficients} and defining $\ell_{n,i}\st{k}=\theta_{n,i}\st{k}/b_{i}\st{k}$ the system \eqref{eqn:GARK-imex22-adj-sol}--\eqref{eqn:GARK-imex22-adj-stage} becomes:
\begin{equation}
\label{eqn:GARK-imex22-adj-scaled}
\begin{split}
&\bigl( \Id_d -h \, \overline{a}_{2,2}\st{\textsc{i},\textsc{i}} \Jac\stran{\textsc{i}}_{n,2} \bigr)\,\ell_{n,2}\st{\textsc{i}} = h \,  \Jac\stran{\textsc{i}}_{n,2}\, \,\lambda_{n+1}, \\
&\ell_{n,2}\st{\textsc{e}} = h \,\Jac\stran{\textsc{e}}_{n,2}\,\left( \lambda_{n+1} + \overline{a}_{2,2}\st{\textsc{e},\textsc{i}} \,\ell_{n,2}\st{\textsc{i}} \right), \\
&\bigl( \Id_d -h \, \overline{a}_{1,1}\st{\textsc{i},\textsc{i}} \Jac\stran{\textsc{i}}_{n,1}\, \bigr)\,\ell_{n,1}\st{\textsc{i}}  = h \, \Jac\stran{\textsc{i}}_{n,1}\, \left(  \lambda_{n+1} +  \overline{a}_{1,2}\st{\textsc{i},\textsc{e}} \,\ell_{n,2}\st{\textsc{e}}  + \overline{a}_{1,2}\st{\textsc{i},\textsc{i}} \,\ell_{n,2}\st{\textsc{i}} \right), \\
&\ell_{n,1}\st{\textsc{e}}  = h \, \Jac\stran{\textsc{e}}_{n,1}\, \left( \lambda_{n+1}
+ \overline{a}_{1,1}\st{\textsc{e},\textsc{i}}\,\ell_{n,1}\st{\textsc{i}} + \overline{a}_{1,2}\st{\textsc{e},\textsc{e}}\,\ell_{n,2}\st{\textsc{e}}
+\overline{a}_{1,2}\st{\textsc{e},\textsc{i}} \,\ell_{n,2}\st{\textsc{i}} \right), \\
& \lambda_{n}  = \lambda_{n+1} + b_{1}\st{\textsc{e}}\,\ell_{n,1}\st{\textsc{e}} + b_{1}\st{\textsc{i}}\,\ell_{n,1}\st{\textsc{i}} + b_{2}\st{\textsc{e}}\,\ell_{n,2}\st{\textsc{e}} + b_{2}\st{\textsc{i}}\,\ell_{n,2}\st{\textsc{i}}.
\end{split}
\end{equation}
Equations \eqref{eqn:GARK-imex22-adj-scaled} represent the particular form the discrete adjoint equations \eqref{eq:GARK_costate} take for our method \eqref{eqn:GARK-imex22}.

%% file: orderconvergence.tex
\section{Order of GARK discrete adjoints}
\label{sec:orderconvergence}

This analysis follows closely that for Runge-Kutta schemes \cite{Sandu_2006_RKdadj}. 
Define the {\it solution sensitivity matrix} of the system \eqref{eqn:split-ode} as:
\begin{equation}
\label{eq:sensitivity-matrix}
S_{t_2,t_1}\bigl(y(t_1)\bigr) \coloneqq \frac{d y(t_2)}{d y(t_1)} \in \mathbb{R}^{d \times d}, \quad t_2 \ge t_1,
\end{equation}
where we made explicit the dependency of the sensitivity matrix on the nonlinear trajectory about which it is computed.

\begin{theorem}[Order of discrete GARK adjoint scheme]
Assume that the system \eqref{eqn:split-ode} is smooth and has a smooth solution, and the goal function \eqref{eq:output_functional} is sufficiently smooth, such that the sensitivity matrix \eqref{eq:sensitivity-matrix} and the Jacobian of \eqref{eq:output_functional} are Lipschitz continuous:
\begin{subequations}
\label{eqn:assumptions}
\begin{equation}
\label{eqn:smoothness-assumption}
\Vert S_{t_2,t_1}\bigl(y\bigr) - S_{t_2,t_1}\bigl(z\bigr) \Vert \le \mathrm{L}\, \Vert y - z \Vert, \ \
\Vert Q_y\bigl(y\bigr) - Q_y\bigl(z\bigr) \Vert \le \mathrm{L}\, \Vert y - z \Vert, 
\end{equation}
for all $y,z$, and $t_2 \ge t_1$.
Assume that the forward GARK scheme \eqref{eq:GARK_scheme} has order $p$, and provides a numerical solution for \eqref{eqn:split-ode} that converges with order $p$:
\begin{equation}
\label{eqn:accuracy-assumption}
y_n = y(t_n) + \mathcal{O}\left( h^p \right).
\end{equation}
\end{subequations}
Then the discrete GARK adjoint solution \eqref{eq:costate_GARK_original} approximates the adjoint ODE solution \eqref{eqn:split-continuous-adj} with the same order of accuracy:
\[
\lambda_n = \overline{\lambda}(t_n) + \mathcal{O}\left( h^p \right).
\]
\end{theorem}

\begin{proof}
Consider the additively-split system  \eqref{eqn:split-ode}.
Infinitesimally small perturbations in the initial conditions of \eqref{eqn:split-ode} evolve in time according to the {\it tangent linear model}:
\begin{equation}
\label{eqn:split-tlm}
\delta y' =  \sum_{q=1}^\ncomp \Jac\st{q}\!(y)\, \delta y,\quad t_0 < t \le t_F, \quad \delta y(t_0)=\delta y_0.
\end{equation}
We have that $\delta y(t_2) = S_{t_2,t_1} \,\delta y(t_1)$, where the {\it solution sensitivity matrix} is defined in \eqref{eq:sensitivity-matrix}.
Note that the formulation \eqref{eqn:split-tlm} of the tangent linear model depends on the solution of the forward system \eqref{eqn:split-ode} $y(t)$, which is the argument where the Jacobians $\Jac\st{q}\left( y \right)$ are evaluated. Consequently the systems \eqref{eqn:split-ode} and \eqref{eqn:split-tlm} need to be advanced together forward in time. 

A numerical solution of  the combined system \eqref{eqn:split-ode}--\eqref{eqn:split-tlm} using a GARK method leads to \eqref{eq:GARK_scheme} and to the following solution of the tangent linear model:
\begin{subequations}
\label{eq:GARK_scheme_tlm}
\begin{align}
\delta y_{n+1} &= \delta y_n + h\, \sum_{q=1}^\ncomp \sum_{i=1}^{s\st{q}} b_i\st{q} \Jac\st{q}\big(Y_{n,i}\st{q}\big) \, \delta Y_{n,i}\st{q}, \label{eq:GARK_step_tlm} \\
\delta Y_{n,i}\st{q} &= \delta y_n + h\, \sum_{m=1}^\ncomp \sum_{j=1}^{s\st{m}} a_{i,j}\st{q,m} \Jac\st{m}\big(Y_{n,j}\st{m}\big) \, \delta Y_{n,j}\st{m},\quad \forall q, \forall i.
\label{eq:GARK_stages_tlm}
\end{align}
\end{subequations}
We have that $\delta y_{n_2} = S^h_{n_2,n_1}(y_{n_1}) \, \delta y_{n_1}$, where the {\it numerical solution sensitivity matrix} is: 
\begin{equation}
\label{eq:numerical-sensitivity-matrix}
S^h_{n_2,n_1}\bigl(y_{n_1}\bigr)\coloneqq \frac{d y_{n_2}}{d y_{n_1}} \in \mathbb{R}^{d \times d}, \quad n_2 \, n_1.
\end{equation}

We note that \eqref{eq:GARK_scheme_tlm} is the GARK scheme applied to the tangent linear model \eqref{eqn:split-tlm}. The stability of the GARK method applied to the tangent linear model is the same as the linear stability of the method applied to the nonlinear model \eqref{eqn:split-ode}. Given the order of accuracy and the convergence with order $p$ assumptions for the method we conclude that the solutions of the tangent linear system also converge with order $p$:
\[
 \delta y_n = \delta y(t_n) + \mathcal{O}\left(h^{p}\right).
\]
If the calculations are carried out starting from the same initial point $y_{n_1} = y(t_{1}) $ and the same unit vector $\delta y_{n_1} = \delta y(t_{1}) = e_i$ we obtain an order $p$ approximation of the $i$-th column of the sensitivity matrix \eqref{eq:sensitivity-matrix}: 
\begin{align}
\mathcal{O}\left(h^{p}\right) &= \delta y(t_{2}) - \delta y_{n_2} = \left( S_{t_{2},t_{1}}\bigl(y(t_{1})\bigr) - S^h_{n_2,n_1}\bigl(y(t_{1})\bigr)\right) \cdot e_i, 
\end{align}
where we assumed that the system is sufficiently smooth and used the assumption that the forward solution converges with order $p$. Since the approximation holds for any column $i$ we have that the numerical sensitivity matrices approximate the continuous ones with the order of the solution:
\begin{equation}
\label{eqn:S-approx-error}
S_{t_{2},t_{1}}\bigl(y\bigr) - S^h_{n_2,n_1}(y) =  \mathcal{O}\left(h^{p}\right) \quad \forall t_2 \ge t_1.
\end{equation}

The {\it adjoint model} corresponding to the system \eqref{eqn:split-ode} and the goal function \eqref{eq:output_functional} is defined as:
\begin{equation}
\label{eqn:split-adj}
\overline{\lambda}' =  -\sum_{q=1}^\ncomp \bigl(\Jac\st{q}\left( y \right) \bigr)^\top  \, \overline{\lambda},\quad t_F > t \ge t_0,\quad \overline{\lambda}(t_F)= \left(\outfull_y\bigl(y(t_F)\bigr)\right)^\top.
\end{equation}
From \eqref{eqn:sensitivity_cts} we have that:
\begin{equation}
\label{eq:continuous-adjoint-propagation}
\overline{\lambda}(t_n) = \left( \frac{d \outfull\bigl(y(t_F)\bigr)}{d y(t_n)} \right)^\top =  \left( \outfull_y\bigl(y(t_F)\bigr) \, \frac{d y(t_F)}{d y(t_n)} \right)^\top
= S_{t_F,t_n}^\top\bigl( y(t_n) \bigr)  \, \overline{\lambda}(t_F),
\end{equation}
and similarly,
\begin{equation}
\label{eq:discrete-adjoint-propagation}
\lambda_n 
= \bigl(S^h_{N,n}( y_n )\bigr)^\top \, \lambda_N = S^h_{N,n}( y_n ) \cdot \outfull_y\bigl(y_N\bigr).
\end{equation}
The difference between the discrete and continuous adjoints at time $t_n$ is:
%
\begin{align*}
&\lambda_n  - \overline{\lambda}(t_n) =  \bigl(S^h_{N,n}( y_n )\bigr)^\top \cdot \left( \outfull_y\bigl(y_N\bigr)  - \outfull_y\bigl(y(t_F)\bigr) \right)^\top \\
&\qquad + \bigl(S^h_{N,n}( y_n ) - S_{t_F,t_n}( y_n ) \bigr)^\top \,  \overline{\lambda}(t_F) 
 + \bigl(S_{t_F,t_n}( y_n ) - S_{t_F,t_n}( y(t_n) ) \bigr)^\top  \, \overline{\lambda}(t_F).
\end{align*}
%
The result follows from the smoothness assumption \eqref{eqn:smoothness-assumption}, the assumed convergence of the forward solution \eqref{eqn:accuracy-assumption}, and from the closeness of discrete and continuous sensitivity matrices \eqref{eqn:S-approx-error}.

\end{proof}

%% file: impactoferrors.tex
\section{Propagation of Discretization Errors in the GARK Framework}
\label{sec:impactoferrors}

In this section we derive estimates of the impact of individual discretization errors (in time or in space) onto the resulting error of the goal function \eqref{eq:output_functional}:
\begin{equation}
\label{eqn:qoi-error}
\mathcal{E} \coloneqq \outfull\bigl(y_N\bigr) - \outfull\bigl(y(t_\F)\bigr).
\end{equation}
To this end, we denote by $y(t)$ the reference solution, which represents either the time continuous solution (for error estimation in time) or a solution on a refined spatial grid, projected onto the current grid (for error estimation in space).

\subsection{Impact of temporal discretization errors}

Consider the GARK scheme \eqref{eq:GARK_scheme} written as a discrete step:
\begin{equation}
\label{eqn:gark-abstract-update}
y_{n+1} = y_{n} + h_n\, \Phi\bigl(h_n,y_n\bigr).
\end{equation}
The corresponding discrete adjoint equation reads:
\begin{equation}
\label{eqn:adjoint-abstract-update}
    \lambda_{n} = \lambda_{n+1} + h_{n}\, \Phi_y^\top(h_{n},y_{n}), \qquad
    \Phi_y(h_{n},y_{n}) \coloneqq \frac{\partial \Phi(h,y) }{\partial y}\Big|_{h=h_{n},y=y_{n}}.
\end{equation} 

Consider now the scheme started from the exact solution $y(t_n)$. It produces a solution $\hat{y}_{n+1}$ that differs from the exact solution by a temporal residual, equal to the local truncation error of the method:
\begin{subequations}
\label{eqn:temporal_residuals}
\begin{align}
\hat{y}_{n+1} - y(t_{n}) - h_n\, \Phi\bigl(h_n,y(t_n)\bigr) &= 0, \\
y(t_{n+1}) - y(t_{n}) - h_n\, \Phi\bigl(h_n,y(t_n)\bigr) &= \rest_{n+1}.
\end{align}
\end{subequations}
%
\begin{proposition}
If the Hessian of the goal function and the discrete step operator $\Phi$ are bounded, i.e., 
\[
\|\frac{d^2}{dy^2} Q(y) \| \leq C_1, \qquad \|\frac{\partial^2}{\partial y^2} \Phi(\cdot,y) \| \leq C_2,
\]
with $C_1,C_2>0$, the error in the goal function \eqref{eqn:qoi-error} is expressed as follows:
\begin{subequations}
\begin{equation}
\label{eqn:temporal-error-impact}
\mathcal{E} = \sum_{n=1}^{\nsteps} \lambda_n^\top  \, \rest_n + \mathcal{O}(h^{2p}),
\end{equation}
where 
\begin{equation}
\label{eqn:time-residuals}
\rest_n = y(t_{n}) - \hat{y}_{n}.
\end{equation}
\end{subequations}
\end{proposition}

\begin{remark}
In \eqref{eqn:temporal-error-impact} the local truncation residuals are of size $\Vert \rest_n \Vert \sim \mathcal{O}(h^{p+1})$. Assuming that the adjoint variables are of moderate size $\Vert \lambda_n \Vert \sim \mathcal{O}(1)$, the a posteriori error estimate is 
\[
\mathcal{E}^{\textsc{est}} = \sum_{n=1}^{\nsteps} \lambda_n^\top \, \rest_n \sim \mathcal{O}(h^{p}),
\]
and the error in this estimate is much smaller than the estimated quantity itself, $\mathcal{E} - \mathcal{E}^{\textsc{est}} \sim \mathcal{O}(h^{2p})$.
\end{remark}

\begin{proof}
The derivation is based on a Taylor expansion of the discrete step operator $\Phi$ and reads
\begin{align*}
\mathcal{E} &= \outfull\bigl(y_N\bigr) - \outfull\bigl(y(t_\F)\bigr) =  Q_y \bigl(y_N\bigr)\,\left( y_N - y(t_\F) \right) + \mathcal{O}\big(h^{2p}\big) \\
 &= \lambda_N^\top \, \left( y_N - y(t_N) \right) + \mathcal{O}\big(h^{2p}\big) \\
 &= \lambda_N^\top \, \left( y_N - \hat{y}_N + \hat{y}_N - y(t_\F) \right) + \mathcal{O}\big(h^{2p}\big) \\
 &= \lambda_N^\top \, \left( y_N - \hat{y}_N \right) + \lambda_N^\top \, \rest_N + \mathcal{O}\big(h^{2p}\big).
\end{align*}
Since 
\begin{equation*}
\Phi\bigl(h_{N-1},y_{N-1}\bigr) - \Phi\bigl(h_{N-1},y(t_{N-1})\bigr) = \Phi_y(h_{N-1},y_{N-1})\,\bigl(y_{N-1} - y\bigl(t_{N-1}\bigr)\bigr) + \mathcal{O}(h^{2 p}),
\end{equation*}
we obtain from \eqref{eqn:adjoint-abstract-update} that
\begin{align*}
\lambda_N^\top\, \bigl( y_N - \hat{y}_N \bigr) 
& \!\!= \!\!\lambda_N^\top\, \Bigl( y_{N-1} -  y(t_{N-1})  + h_{N-1}  \bigl( \Phi\bigl(h_{N-1},y_{N-1}\bigr) - \Phi(h_{N-1},y(t_{N-1})) \bigr)  \Bigr) \\
& = \lambda_N^\top\, \Bigl(\Id_d + h_{N-1}\,\Phi_y(h_{N-1},y_{N-1}) \Bigr)  \Bigl( y_{N-1} -  y(t_{N-1}) \Bigr) + \mathcal{O}(h^{2 p}) \\
&= \lambda_{N-1}^\top\,  \Bigl( y_{N-1} -  y(t_{N-1}) \Bigr) + \mathcal{O}(h^{2 p}), 
\end{align*}
where we have used the adjoint update step \eqref{eqn:adjoint-abstract-update}. Therefore
\begin{align*}
\mathcal{E} 
 &=  \lambda_{N-1}^\top\,  \Bigl( y_{N-1} -  y(t_{N-1}) \Bigr)  + \lambda_N^\top\,  \rest_N + \mathcal{O}\big(h^{2p}\big) \\
 &=  \lambda_{N-1}^\top\,  \Bigl( y_{N-1} -  \hat{y}_{N-1} \Bigr)  + \lambda_{N-1}^\top \,  \rest_{N-1} + \lambda_N^\top \,  \rest_N + \mathcal{O}\big(h^{2p}\big).
\end{align*}
Repeating the argument in a recursive manner yields the desired result
\begin{equation}
    \mathcal{E} = \lambda_0^\top \, \underbrace{(y_0 - y(t_0))}_{=0} + \sum_{n=1}^N \lambda_n^\top \, \rest_n + \mathcal{O}\big(h^{2p}\big).
\end{equation}
A similar recursive argument has been used in the context of multirate ODEs in \cite{estep2012posteriori}.
\end{proof}
The accumulation of temporal errors is illustrated in Figure \ref{fig:temporal-error}. Note that each adjoint variable $\lambda_n$ is evaluated along a different numerical solution started at the exact solution $y(t_{n-1})$. In practice we compute all adjoint variables along the same numerical solution started at $y(t_{0})$.


%
\begin{figure}
\centering
\includegraphics[width=0.75\textwidth]{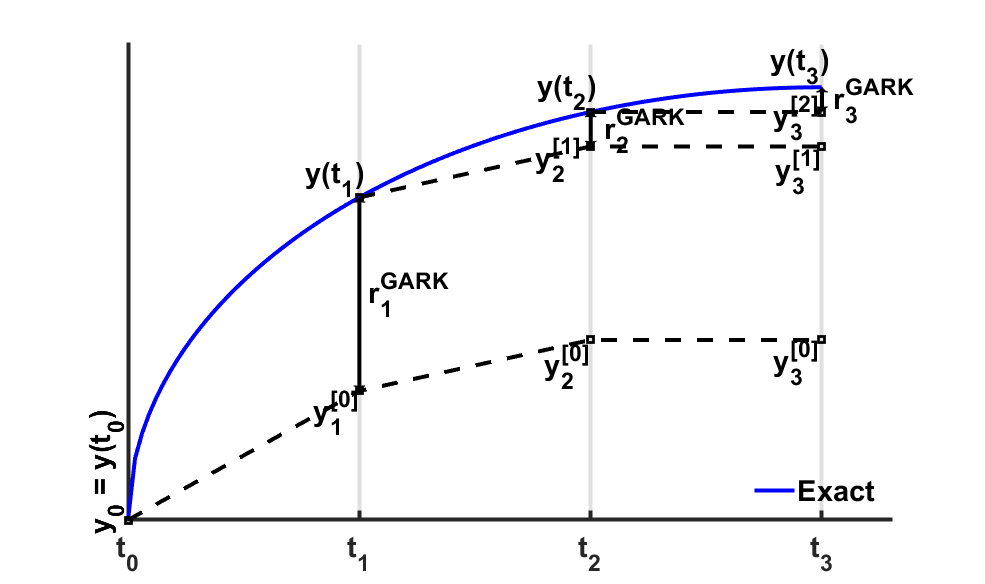}
\caption{Accumulation of temporal errors. Different numerical solutions, each started at the exact $y(t_n)$ for a different $n$, are shown.}
\label{fig:temporal-error}
\end{figure}

\subsection{Influence of spatial discretization errors}

The methodology put forth in this paper is not restricted to a specific type of differential equation. For the clarity of presentation we consider the advection-diffusion-reaction equation in more detail, as it allows for explicit examples of residuals and the obtained error estimates. The computational domain is denoted as $\Omega$ and is assumed to be sufficiently regular. In particular, we assume that its boundary can be well-approximated by polyhedrons.  
The advection-diffusion equation in continuous form reads
\begin{subequations}
\label{eqn:advection-diffusion-continuous}
\begin{align}
    u'(t,x) &= k(x)\,\Delta\,u(t,x) + a(x) \cdot \,\nabla\,u(t,x) + \phi(t,x), \  &&\forall (x,t) \in \Omega \times (t_0,t_{\text{F}}], \\ 
    u(t_0,x) &= u_0(x), \ &&\forall x \in \Omega, \\
    u(t,x) &= 0, \  &&\forall (x,t) \in \partial \Omega \times (t_0,t_{\text{F}}].
\end{align}
\end{subequations}
This equation is a prototype for a multiphysics process, where reaction, diffusion and advection phenomena may take place on largely different temporal and spatial scales. The following discussion is based on the space
\[
\mathcal{H} = H^1_0(\Omega) = \{u \in L^2(\Omega) \ | \ \nabla u \in (L^2(\Omega))^3, u |_{\partial \Omega} = 0 \}.
\]
%
We seek for solutions $u(t) \in \mathcal{H}$ subject to
\begin{multline}
    \label{eqn:advection-diffusion-weak}
    \int_{\Omega} u'(t,x)\, v(x) \, dx = \int_{\Omega} k(x)\,\nabla\,u(t,x)\cdot\,\nabla\,v(x)\,dx  + \\
    \int_{\Omega} a(x)\cdot \,\nabla\,u(t,x)\,v(x)\,dx + \int_{\Omega} \phi(t,x)\,v(x)\,dx,
\end{multline}
for all test functions $v \in \mathcal{H}$. We consider an orthonormal basis $\{\varphi_k\}_{k=1}^\infty$ of $\mathcal{H}$. The exact solution of \eqref{eqn:advection-diffusion-weak} can then be written as
\begin{equation}
    u(t,x) = \sum_{k=1}^\infty u_k(t) \, \varphi_k(x)
\end{equation}
and a FEM discretization in a Galerkin setting is obtained by approximating the solution by the truncated sum
\begin{equation}
    u(t,x) \approx u^{\textsc{fem}}(t,x) := \sum_{k=1}^{N^{\textsc{fem}}} y_k(t) \,\varphi_k(x).
\end{equation}
We use the same symbol $u$ for both the continuous solution $u(t,x)$ and the vector $u = (u_1,u_2,\ldots,u_{N^{\textsc{fem}}})^\top$, as the meaning can always be inferred from the specific context. 
The FEM solution $u^{\textsc{fem}}$ is identified with the vector $y=(y_1,y_2,\ldots,y_{N^{\textsc{fem}}})^\top$. When the ODE \eqref{eqn:ode} is the result of such a FEM discretization it has an implicit form, where the left-hand-side is multiplied by a non-singular mass matrix as
\begin{equation}
\label{eqn:ode-implicit}
\Mass\,y' = \fun^{\textsc{fem}}(t,y), \quad y(t_0)=y_0,
\end{equation}
where the right-hand-side is given as $\fun^{\textsc{fem}}(\cdot,y) = \Diffusion y + \Advection y + \source$. In this expression, the mass, diffusion, and advection matrices have  entries
\begin{align*}
m_{i,k} &:= \int_{\Omega}\varphi_k(x) \,\varphi_i(x)\,dx, \\    
d_{i,k} &:= \int_{\Omega}k(x)\,\nabla \varphi_k(t,x)\,\cdot\,\nabla \varphi_i(x) \,dx, \\
a_{i,k} &:= \int_{\Omega} \,a(x)\cdot\nabla \varphi_k(t,x)\,\varphi_i(x)\,dx, \qquad i,k= 1,\ldots, N^{\textsc{fem}}.
\end{align*}
Different partitions lead to different forms of the implicit ODE \eqref{eqn:ode-implicit}. When the exact PDE solution $u(t,x)$ is inserted into the FEM discretization, equation \eqref{eqn:ode-implicit} has a remaining residual {\it coming from the spatial discretization only}:
\begin{equation}
\label{eqn:ode-implicit-residual}
\begin{split}
\Mass\,u' &= \fun^{\textsc{fem}}(t,u) + \resfem(t,u), \\
u' &= \Mass^{-1}\,\fun^{\textsc{fem}}(t,u) + \Mass^{-1}\,\resfem(t,u) = \fun(t,u) + \resx(t,u).
\end{split}
\end{equation}
Indeed, if we combine 
\begin{equation}
    \sum_{k=1}^\infty m_{i,k} \,u_k' = \sum_{k=1}^\infty d_{i,k}\,u_k(t) + \sum_{k=1}^\infty a_{i,k}\,u_k(t) + s_i(t), \quad i=1,\dots,N^{\textsc{fem}}, 
\end{equation}
with its truncated counterpart 
\begin{equation}
    \sum_{k=1}^{N^{\textsc{fem}}} m_{i,k} \,y_k' = \sum_{k=1}^{N^{\textsc{fem}}} d_{i,k}\,y_k(t) + \sum_{k=1}^{N^{\textsc{fem}}} a_{i,k}\,y_k(t) + s_i(t), \quad \quad i=1,\dots,N^{\textsc{fem}}, 
\end{equation}
then we obtain 
\begin{multline}
\label{eqn:residual_ADR_fem}
\sum_{k=1}^{N^{\textsc{fem}}} m_{i,k} \,u_k' = \sum_{k=1}^{N^{\textsc{fem}}} d_{i,k}\,u_k(t) +  \sum_{k=1}^{N^{\textsc{fem}}} a_{i,k}\,u_k(t) + s_i(t) \\
 \quad \underbrace{ - \sum_{k=N+1}^\infty \left( m_{i,k} \,u_k' - d_{i,k}\,u_k(t) - a_{i,k}\,u_k(t) \right) }_{\resfem_i(t,u)}.
\end{multline}


\subsubsection{Spatial discretization errors for finite difference and finite volume schemes}
%
To fix ideas, we will additionally analyze the form of the residuals in finite volume/finite difference form. For this we assume that the system \eqref{eqn:advection-diffusion-continuous} admits a smooth solution  $u  \in C^\infty \left( \Omega \times [t_0,t_{\text{F}}] \right)$. Let $\overline{x}$ be the grid (set of grid point coordinates) over which the continuous PDE is discretized using finite differences. The semi-discretized system is written as an ODE as follows:
\begin{equation}
\begin{split}
    \label{eqn:advection-diffusion-finite-difference}
    y'(t,\overline{x}) &= \mathbf{D}(\overline{x})\,y(t,\overline{x}) + \mathbf{A}(\overline{x})\,y(t,\overline{x}) + \phi(t,\overline{x}), \quad y(t_0,\bar{x}) = u_0(\bar{x}),
\end{split}
\end{equation}
where $\mathbf{D}$ is the discrete operator corresponding to the diffusion term, and $\mathbf{A}$ the discrete operator corresponding to the advection term. If we now substitute the exact (smooth) solution $u(t,x)$ in the semi-discretized equation above, we get the following residual over the grid $\overline{x}$:
\begin{equation}
\begin{split}
u'(t,\overline{x}) &= \mathbf{D}(\overline{x})\,u(t,\overline{x}) +  \mathbf{A}(\overline{x})\,u(t,\overline{x}) + \phi(t,\overline{x}) \\
&\quad + \underbrace{k(x)\,\Delta\,u(t,x) \big|_{x=\bar{x}} - \mathbf{D}(\overline{x})\,u(t,\overline{x})}_{\resx_{\textsc{diff}}(t,\overline{x},u)} \\
&\quad + \underbrace{a(x)\,\nabla\,u(t,x) \big|_{x=\bar{x}}- \mathbf{A}(\overline{x})\,u(t,\overline{x})}_{\resx_{\textsc{adv}}(t,\overline{x},u)}, \quad u(t_0,\bar{x}) = u_0(\bar{x}).
\end{split}
\end{equation}
We see that $\resx_{\textsc{diff}}(t,\overline{x},u)$, $\resx_{\textsc{adv}}(t,\overline{x},u)$ are the familiar residuals obtained in the analysis of finite volume/difference schemes.

%
\subsubsection{Spatial discretization errors for additively-split multiphysics systems}

Before deriving error representations for the spatial discretization in the next subsection, we illustrate the structure of residuals for several representative examples of additively-split systems. 

Consider a PDE coupled with a source term. Let $\Mass\st{1},\fun\st{1}$ be the mass matrix and the discrete PDE terms, and $\fun\st{2}$ the corresponding source terms. The FEM discretization of the PDE part only leads to:
\begin{equation}
\begin{split}
\Mass\st{1}\,y' &= \fun\st{1}(t,y)  + \resfem\,\st{1}(t,u) + \Mass\st{1}\,\fun\st{2}(y), \quad y(t_0) = y_0, \\
y' &= \Mass\st{1}\,^{-1}\,\fun\st{1}(t,y) + \resx\,\st{1}(t,u) + \fun\st{2}(t,y).
\end{split}
\end{equation}
For two coupled PDE simulations that exchange fluxes/forces $\varphi$ through the common boundaries, e.g., an elastic wing and the flow around it, the system would be:
\begin{equation}
\begin{split}
\Mass\st{1}\,y' &= \fun\st{1}(y) + \resfem\,\st{1}(t,u) + \varphi\st{1}(y,z), \quad y(t_0) = y_0, \\
\Mass\st{2}\,z' &= \fun\st{2}(z) + \resfem\,\st{2}(t,u) + \varphi\st{2}(y,z),\quad z(t_0) = z_0.
\end{split}
\end{equation}
For a domain-partitioned PDE simulation, which is important for parallelization of large-scale systems, the structure takes the form:
\begin{equation}
\begin{split}
\Mass\st{1,1}\,y' + \Mass\st{1,2}\,z'  &= \fun\st{1}(y,z) + \resfem\,\st{1}(t,u), \\
\Mass\st{2,1}\,y' + \Mass\st{2,2}\,z'  &= \fun\st{2}(y,z) + \resfem\,\st{2}(t,u).
\end{split}
\end{equation}
%
%

\subsection{Influence of spatial discretization errors}

Inserting the exact solution in the spatially semi-discrete equations of the additively-split system \eqref{eqn:split-ode} leads to the following dynamics:
\begin{equation}
\label{eqn:split-ode-perturbed}
u' =  \sum_{q=1}^\ncomp \left( \fun\st{q}(u) + \resx\,\st{q}(u)\right), 
\end{equation}
where the discretization of each component process introduces a different numerical residual $\resx\,\st{q}(u)$. We assume that these residuals can be accurately estimated along each numerical solution, and therefore we treat them as known functions of their arguments. We also assume that these residuals remain very small throughout the simulation.

Application of the GARK scheme \eqref{eq:GARK-alt_scheme} to solve the perturbed system \eqref{eqn:split-ode-perturbed} gives:
\begin{subequations}
\label{eq:GARK-alt_scheme-pert}
\begin{align}
\label{eq:GARK-alt_step-pert}
\widehat{y}_{n+1} &= \widehat{y}_n + h_n \sum_{q=1}^\ncomp \sum_{i=1}^{s\st{q}} b_i\st{q}\, \widehat{k}_{n,i}\st{q}, \\
\label{eq:GARK-alt_stages-pert}
\widehat{k}_{n,i}\st{q} &= \fun\st{q}\Bigl(T_{n,i},y_n + h_n \sum_{m=1}^\ncomp \sum_{j=1}^{s\st{m}} a_{i,j}\st{q,m}\, \widehat{k}_{n,j}\st{m}\Bigr) + \resx\,\st{q}_{n,i},\\
\nonumber
&\qquad q=1,\ldots,\ncomp, \quad  i=1,\ldots,s\st{q},
\end{align}
\end{subequations}
where we denote the residuals estimated along the numerical solution by
\begin{equation}
\label{eqn:space_residuals}
\resx\,\st{q}_{n,i} \coloneqq \resx\,\st{q}\Bigl(T_{n,i},\widehat{y}_n + h_n \sum_{m=1}^\ncomp \sum_{j=1}^{s\st{m}} a_{i,j}\st{q,m}\, \widehat{k}_{n,j}\st{m}\Bigr).
\end{equation}

\begin{proposition}
To leading order, the error in the goal function \eqref{eq:output_functional} is approximated as follows:
\begin{equation}
\label{eqn:spatial-error-impact}
\begin{split}
\mathcal{E} 
&\approx \sum_{n=1}^{\nsteps}  \sum_{q=1}^\ncomp \mu_{n,i}\st{q}\,^\top \, \resx\,\st{q}_{n,i},
\end{split}
\end{equation}
where the adjoint variables $\mu_{n,i}\st{q}$ are the ones discussed in Comment \ref{com:alternative-GARK-adjoint}.
\end{proposition}

\begin{proof}
The proof follows immediately from the interpretation of Lagrange multipliers as sensitivities \eqref{eqn:mu-as-sensitivity}. The spatial discretization residuals  are small perturbations of the stage slope equations \eqref{eq:GARK-alt_stages-pert}. The impact of a single spatial discretization error $\resx\,\st{q}_{n,i}$ is:
\[
\mathcal{E} \approx \frac{d Q(y_{N})}{d k_{n,i}\st{m}}\, \resx\,\!\st{q}_{n,i} = \mu_{n,i}\st{m}\,\!^\top\,  \resx\,\st{q}_{n,i}.
\]
The total spatial error impact \eqref{eqn:spatial-error-impact} is obtained by adding up the contributions of all residuals in all stages and time steps.
\end{proof}

%% file: implementation.tex
\section{Implementation details}
\label{sec:implementation}
All the algorithms discussed in the following were implemented in \texttt{python}, using the \texttt{Fenics} \cite{AlnaesBlechta2015a,LoggMardalEtAl2012a} library to model PDEs and the Gryphon framework \cite{ErikSkare2012} for time-stepping inside \texttt{Fenics}.

We use the second-order implicit-explicit GARK method discussed in Section \ref{subsec:discrete_adjoint_GARK_example}, which is a variant of \cite[Example 9]{Sandu_2015_GARK}, with coefficients
%
%
$\gamma = 1 - \sqrt{2}/2$ and $\alpha = \gamma$ for all numerical experiments. We next discuss several aspects of the underlying implementation, where we use the notation $\Delta t \equiv h$, to clearly distinguish between the temporal resolution $h$ and the space resolution $\Delta x$.

\subsection{Different solutions involved in the procedure}

We assume that the base numerical solution (henceforth, numerical solution) is computed with a local time and space resolution of ($\Delta t \equiv h$, $\Delta x$). Reference solutions are obtained by refining the grids in space, in time, and in both dimensions. The effort required to calculate these solutions is exemplified in Figure \ref{fig:reference_solutions} with the distance from the origin $O$ as the relative measure. 

\begin{figure}
    \centering
    \input{reference-solutions.tex}
    \caption{Reference, space- and time-refined solutions, as well as numerical solution with associated spatial and temporal resolutions. The approximate computational effort, representing the actual effort for a specific 2D problem, is represented by the distance to the origin $O$. Temporal and spatial resolution are denoted as ($\Delta t \equiv h$, $\Delta x$).}
    \label{fig:reference_solutions}
\end{figure}
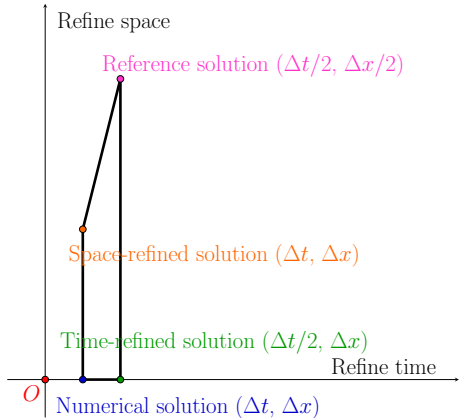

To simplify our implementation, at the start of each experiment, we determine a timestep schedule for the numerical solution. The timestep schedule of the reference solution and the time-refined solution is obtained by dividing each interval of the numerical solution in half. Note that allowing adaptive timesteps together with a variable grid size, as is done in practice, leads to an interaction of numerical errors that can complicate the validation of adjoint error estimation results. 

The spatial grid is generated in a similar manner, where we start with a uniform grid for the numerical solution, and obtain the grid for the reference and space-refined solution via uniform refinement by a factor of two in each spatial dimension. 

The reference, time-refined and space-refined solutions are each computed independent of the numerical solution and they are stored and recalled as needed to compute the residuals as described in the following subsections. Then, adjoint sensitivities are computed for the numerical solution and used with the residuals to construct error estimates.

\subsection{Estimation of error residuals}


We use the time-refined solution as a proxy for the exact solution to compute the time residuals (local truncation error) via \eqref{eqn:temporal_residuals}. For each timestep of the numerical solution, $t_n$ to $t_{n + 1} = t_n + \Delta t$, we compute the local truncation error by taking a full $\Delta t$ step while starting at the corresponding time-refined solution. Once the residuals are computed, we use the adjoint information obtained using the numerical solution to compute the time error according to \eqref{eqn:temporal-error-impact}. 


To estimate the spatial residuals, we plug in the $y_n$ and $k$-stage values from the space-refined solution projected onto the numerical grid in \eqref{eq:GARK-alt_stages-pert}. The spatial errors are computed by plugging in the information from the adjoints and the residuals into \eqref{eqn:spatial-error-impact}. We compute a separate spatial error for each partition of the GARK method. 

Note that computing a computationally expensive space-time reference solution may not be required in practice. Indeed, computing a space-refined solution could be avoided for instance by applying recovery methods, see \cite{Zienkiewicz1992}, or other post-processing techniques. Also, embedded methods provide time-accurate solutions which can be used to estimate local truncation errors efficiently. We explicitly compute a space-time reference solution to remove any influence from post-processing steps so that we can verify the adjoint error estimation framework presented in the paper. 


%% file: reference-solutions.tex
\definecolor{ffqqqq}{rgb}{1.,0.,0.}
\definecolor{ffttcc}{rgb}{1.,0.2,0.8}
\definecolor{ffwwqq}{rgb}{1.,0.4,0.}
\definecolor{qqzzqq}{rgb}{0.,0.6,0.}
\definecolor{qqqqcc}{rgb}{0.,0.,0.8}
\begin{tikzpicture}[scale=0.5,line cap=round,line join=round,>=triangle 45,x=1.0cm,y=1.0cm]
\begin{axis}[ticks=none,
x=1.0cm,y=1.0cm,
axis lines=middle,
xmin=-1.0,
xmax=11.0,
ymin=-1.0,
ymax=10.0,
xtick={-1.0,0.0,...,10.0},
ytick={-1.0,0.0,...,10.0},]
\clip(-1.,-1.) rectangle (11.,10.);
\draw [line width=2.pt] (1.,0.)-- (2.,0.);
\draw [line width=2.pt] (2.,0.)-- (2.,8.);
\draw [line width=2.pt] (2.,8.)-- (1.,4.);
\draw [line width=2.pt] (1.,4.)-- (1.,0.);
\begin{scriptsize}
\draw [fill=qqqqcc] (1.,0.) circle (2.5pt);
\draw[color=qqqqcc] (3.8,-0.75) node {\LARGE Numerical solution ($\Delta t$, $\Delta x$)};
\draw [fill=qqzzqq] (2.,0.) circle (2.5pt);
\draw[color=qqzzqq] (4.5,1) node {\LARGE Time-refined solution ($\Delta t/2$, $\Delta x$)};
\draw [fill=ffwwqq] (1.,4.) circle (2.5pt);
\draw[color=ffwwqq] (4.5,3.3) node {\LARGE Space-refined solution  ($\Delta t$, $\Delta x$)};
\draw [fill=ffttcc] (2.,8.) circle (2.5pt);
\draw[color=ffttcc] (5.530960823518953,8.358934316665255) node {\LARGE Reference solution ($\Delta t/2$, $\Delta x/2$)};
\draw [fill=ffqqqq] (0.,0.) circle (2.5pt);
\draw[color=ffqqqq] (-0.35862629968514304,-0.35862629968514304) node {\LARGE $O$};
\draw (9, 0.35862629968514304) node {\LARGE Refine time};
\draw (1.82, 9.5) node {\LARGE Refine space};
\end{scriptsize}
\end{axis}
\end{tikzpicture}

%% file: numerics.tex
\section{Numerical Results}
\label{sec:numerics}

We perform three types of experiments on reaction-diffusion problems:

\begin{itemize}
    \item[1.] \textit{Time convergence} experiments validate the implementation of the time integration methods and the sensitivity computations on a reaction-diffusion system from literature \cite{Calvo_2001}.
    \item[2.] \textit{A posteriori error estimation} experiments assess the quality of the proposed adjoint-based error estimators against the difference between the goal function values computed using the space-time reference and the numerical solutions.
    \item[3.] \textit{Adaptive space-time refinement} experiments use the computed adjoint error estimators to locally refine the space and the time grids based on criteria described in the subsections that follow.
\end{itemize} 

Throughout, we consider the goal function to be a linear function of the solution at the end of the time interval that corresponds to the integral of the solution over the entire spatial domain ($\Omega$). It is defined as follows:
\begin{equation}
\label{eq:qoi-numerical-experiments}
\Psi(y_0) = Q(y_N) = \sum_{k=1}^{N^{\textsc{fem}}} \int_{\Omega} \bigl(y_N\bigr)_k \,\varphi_k(x) \, dx. 
\end{equation}
Furthermore for each problem under consideration, we additively split the right-hand side into two components as follows:
\[
y' = f\st{1}(y) + f\st{2}(y),
\]
where $f\st{1}$ represents the FEM discretization of the diffusion part, parallel to the description in Section~\ref{sec:impactoferrors}, and $f\st{2}$ represents the FEM discretization of the reaction part. We use Lagrange polynomials of degree two (except for the Gray-Scott equations where we use degree one) as the basis for finite element discretization of each problem. In the following, when we report spatial errors for each partition, they correspond to the contribution from each partition to the error in the goal function.

\subsection{A simple reaction-diffusion system}

Consider the reaction diffusion problem from \cite{Calvo_2001}:
\begin{subequations}
\label{eqn:reaction-diffusion-calvo}
\begin{equation}
\begin{split}
& \frac{\partial u}{\partial t} = \nu \Delta u + u -u^3 +f, \\
& \Omega \equiv (x,y) \in [-1,3]\times[-1,1]~\textnormal{(space units)}, \quad t \in [0, 1.5]~\textnormal{(time units)},
\end{split}
\end{equation}
where the forcing term ensures that the exact solution of \eqref{eqn:reaction-diffusion-calvo} is 
\begin{equation}
\begin{split}
u(t,x,y) &= \begin{cases}
\frac{1}{30}\,(2 + \cos(\pi t)) \, (x+1)(2x - \frac{21}{4})\, (y^2 - 1), & -1 \leq x \leq 2,\\[3pt]
\frac{1}{30}\,(2 + \cos(\pi t)) \, (3-x)(x - \frac{23}{4})\, (y^2 - 1),  & 2 \leq x \leq 3.
\end{cases}
\end{split}
\end{equation}
\end{subequations}
The initial condition is obtained from the assumed true solution and homogenous Dirichlet boundary conditions are used. The spatial grid resolution is $40 \times 20$. For this system, we take  $f\st{1}$ to be the FEM discretization of $\nu \Delta u$ and $f\st{2}$ that of  $u - u^3 + f$.

To assess the convergence of the forward and adjoint methods, we compute the reference solution using a time step of $\Delta t_{\textsc{ref}} = 0.15 \times 2^{-7}$ in place of the analytic solution, while the numerical solutions are computed with $\Delta t = 0.15 \times 2^{-7}$, $k=-4,\dots,0$. Figure \ref{fig:fixed_step_convergence} demonstrates that both the primal and the adjoint solutions converge with order two, as expected. The adjoint convergence order, in particular, is in agreement with the theory for discrete GARK adjoints developed in Section \ref{sec:orderconvergence}. 

\begin{figure}[t!]
	\centering
	\pgfplotstableread{Figures/data_fixed_step_convergence.txt}{\fixedConvergence}
	\begin{tikzpicture}[scale=0.85]
	\begin{loglogaxis}[xlabel=Time stepsize $\Delta t \equiv h $,ylabel=Relative $l^2$-error,ymajorgrids=true,xmajorgrids=true,legend pos = outer north east]
	\addplot[color=blue,mark=x,line width=1.25pt] table [y=y_err, x=dt_vals]{\fixedConvergence};
	\addplot[color=magenta,mark=o,line width=1.25pt] table [y=l_err, x=dt_vals]{\fixedConvergence};
	\legend{Forward,Adjoint}
	\end{loglogaxis}
	\end{tikzpicture}
	\caption{Convergence of forward and adjoint time integrators on the reaction-diffusion system \eqref{eqn:reaction-diffusion-calvo}, evaluated in the Eucledian norm ($l^2$).}
	\label{fig:fixed_step_convergence}
\end{figure}
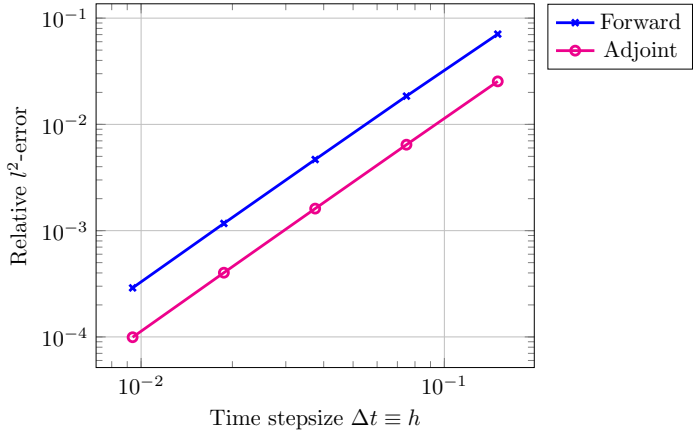

\subsection{Gray-Scott Equations}

The Gray-Scott equations \cite{gray1984}
\begin{subequations}
\label{eqn:gray-scott}
\begin{equation}
	\begin{split}
	& \parfrac{u}{t} = d_u \Delta u  -  u v^2 + f(1 - u),\qquad
	\parfrac{v}{t} = d_v \Delta v  +  u v^2 - (f + k)v, \\
	& \Omega \equiv (x,y) \in [0,2]\times[0,2]~\textnormal{(space units)}, \quad t \in [0, 50]~\textnormal{(time units)},
	\end{split}
\end{equation}
model the diffusion and reaction of two species $u$ and $v$ engaged in the following chemical reactions:
\begin{equation}
	\begin{split}
	U + 2V &\rightarrow P  ,\qquad
	V \rightarrow P.
    \end{split}
\end{equation}  
%
In \eqref{eqn:gray-scott}, $u$ and $v$ are the concentrations of the two chemical species, $d_u$ and $d_v$ are the coefficients of diffusion of the species, $f$ is the feed rate of $u$, and $f + k$ is drain rate of $v$. The parameter values $f = 0.024$, $k = 0.06$, $d_u = \num{8e-2}$ and $d_v = \num{4e-2}$ are used.
The initial conditions for the model are: 
\begin{equation}
\begin{split}
u(x, y, 0) &= \begin{cases}
1 - 2v & 0.75 \leq x \leq 1.25 \\
0 & \textnormal{otherwise},
\end{cases} \\
v(x, y, 0) &= \begin{cases}
0.25\, \sin^2(4 \pi x)\, \sin^2(4 \pi y) & 0.75 \leq x \leq 1.25 \\
1 & \textnormal{otherwise}.
\end{cases}
\end{split}
\end{equation}
\end{subequations}
We assume homogeneous Neumann  boundary conditions. Spatial grids of dimension $20 \times 20$ and $10 \times 10$ are used for the reference and the numerical grid, respectively. Likewise, $\Delta t_{\textsc{ref}} = 0.01$ and $\Delta t_{\textsc{num}} = 0.02$ are used for the reference and numerical timestep schedule, respectively. Moreover, we split the system along the two physics with $f\st{1}$  representing the FEM discretization of the stacked diffusion terms $[d_u \Delta u ; d_v \Delta v]$ and $f\st{2}$ representing that of the stacked reaction terms $[- u v^2 + f(1 - u); u v^2 - (f + k) v]$.

Table \ref{tab:goals_and_errors} shows the goal function values and the adjoint-based estimates for temporal and spatial errors. We observe that the sum of the error measures predicts the difference between the goal functions to within 12\%. For the given settings, the space error arising from the diffusion part contributes most to the overall estimated error as shown in Table \ref{tab:goals_and_errors}.
\begin{figure}
    \centering
    \includegraphics[scale=0.26]{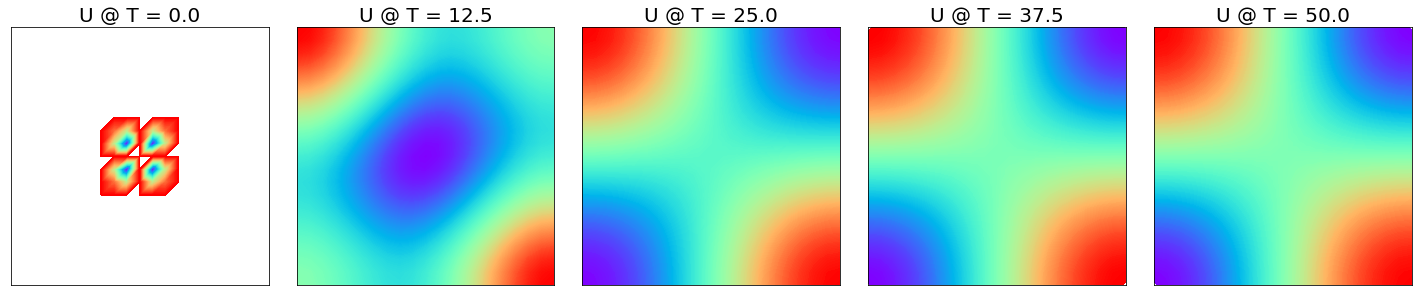}
    \caption{Snapshots of the solution of the Gray-Scott equations \eqref{eqn:gray-scott} on the reference space and time grid.}
    \label{fig:grayscott_solution_snapshot}
\end{figure}



\subsection{BSVD}

The BSVD reaction-diffusion PDE  \cite{heineken2006partitioning} is:
\begin{subequations}
\label{eqn:BSVDeqn}
\begin{equation}
\frac{\partial u}{\partial t} = \nabla \cdot (D(x,y) \nabla u) + 10\, (1 - u^2)(u + 0.6), 
\end{equation}
where 
\begin{equation}
\label{eqn:BSVDeqn_domain}
    \Omega \equiv (x,y) \in [0, 1] \times [0, 1]~\textnormal{(space units)}, \quad t \in [0, 7]~\textnormal{(time units)}.
\end{equation}
The space-dependent diffusion coefficient in \eqref{eqn:BSVDeqn} is defined as
\begin{equation}
D(x,y) = \frac{1}{10}\sum_{i = 1}^{3} e^{-100((x - 0.5)^2 + (y - y_i)^2)},
\end{equation}
whereas the initial condition is given by
\begin{equation}
u(x, y, 0) = 2e^{-10((x - 0.5)^2 + (y + 0.1)^2)} - 1,
\end{equation}
with $y_i$ for $i = 1,2,3$ specified as $y_1 = 0.6$, $y_2 = 0.75$ and $y_3 = 0.9$. Again, homogeneous Neumann boundary conditions are applied, i.e.
\begin{equation}
D(x,y)\, \nabla u(x,y,t)  \cdot  \hat n(x, y) = 0, 
\end{equation}
\end{subequations}
on the boundary, where $\hat n(x, y)$ is the normal vector pointing outwards at the boundary. The spatial resolution is $80 \times 80$ for the reference grid and $40 \times 40$ for the numerical grid and the reference timestep is $\Delta t_{\textsc{ref}} = 0.005$, whereas the numerical timestep is $\Delta t_{\textsc{num}} = 0.01$. The system is again split along the two physics with $f\st{1}$ being the FEM discretization of the diffusion term, $\nabla \cdot (D(x,y) \nabla u)$, and $f\st{2}$ representing the FEM discretization of the reaction term, $10\, (1 - u^2)(u + 0.6)$.
\begin{table}[htb!]
    \footnotesize
    \centering
    \begin{tabular}{|c|c|c|}
        \hline
          & Gray-Scott Equations \eqref{eqn:gray-scott} & BSVD \eqref{eqn:BSVDeqn}  \\ \hline
        Goal function $\Psi_\text{ref}$  & \num{3.9906E+00} &	\num{1.5228E-01} \\ \hline  
        Ref. error  $\mathcal{E}_\text{ref} = \Psi_\text{ref} - \Psi_{num}$& \num{-6.5934E-03} & \num{-2.4749E-01}  \\ \hline 
        Time error ($\mathcal{E}_1$) & \num{-4.2056E-07} &	\num{-5.9835E-04} \\ \hline  
        Space error (diffusion $\mathcal{E}_2$) & \num{-5.8390E-03} &	\num{-7.5768E-04}  \\ \hline  
        Space error (reaction $\mathcal{E}_3$) & \num{-7.0508E-06} &	\num{-7.2872E-02} \\ \hline  
        Estimated error ($\mathcal{E} = \sum_i \mathcal{E}_i$) & \num{-5.8465E-03} &	\num{-7.4228E-02} \\ \hline
        Estimation accuracy ($(\mathcal{E} - \mathcal{E}_\text{ref})/\mathcal{E}_\text{ref})$ & \hphantom{+}\num{1.1328E-01} & \hphantom{+}\num{7.0008E-01} \\ \hline
    \end{tabular}
    \caption{Goal function values and error estimates rounded to five digits on successively refining space and time grids on Gray-Scott \eqref{eqn:gray-scott} and BSVD equations \eqref{eqn:BSVDeqn}.}
    \label{tab:goals_and_errors}
\end{table}

Figure \ref{fig:bsvd_solution_snapshot} shows the solution to the BSVD equation \eqref{eqn:BSVDeqn} on the space-refined and time-refined grid. Table \ref{tab:goals_and_errors} shows the goal function values and estimates of errors of each type. We observe that the sum of the error measures predicts the difference between the goal functions to within {70\%} of the actual error. Even though the difference is rather high, the error measures have the same sign. Moreover, when we perform adaptive mesh refinement in the next subsection, we notice that numerical errors are estimated with increasing accuracy. This is expected, as the numerical solution is employed instead of the exact solution in the adjoint systems of equations, see Section~\ref{sec:impactoferrors}.

\begin{figure}[htb!]
    \centering
    \includegraphics[scale=0.26]{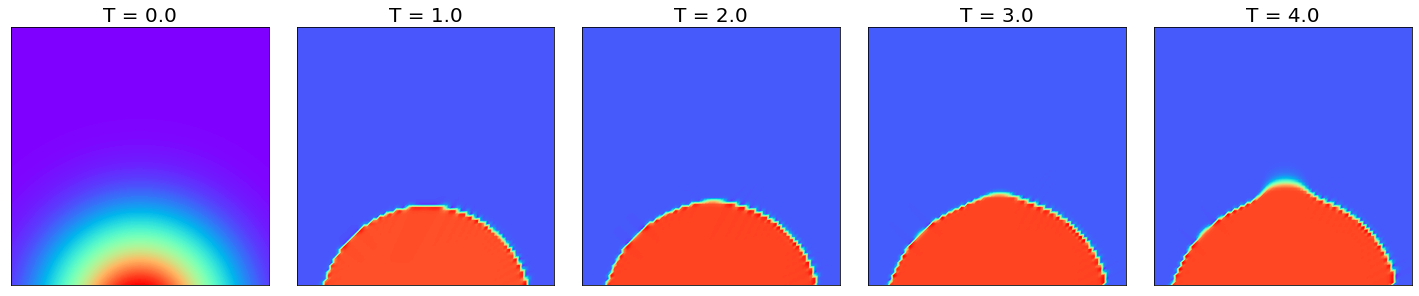}
    \caption{Snapshots of the BSVD solution \eqref{eqn:BSVDeqn} on the reference space and time grid.}
    \label{fig:bsvd_solution_snapshot}
\end{figure}



\subsubsection{Goal-oriented adaptive refinement for the BSVD problem} 
We perform adaptive space-time refinement experiments on the BSVD equation \eqref{eqn:BSVDeqn}. We limit the time span to $[0, 4]$ time units, i.e., we integrate to the point where a small protrusion forms at the front of the solution as shown in Figure \ref{fig:bsvd_solution_snapshot}.  We start with uniform spatial grids (40$\times$40 for the reference and space-refined solutions, and 20$\times$20 for the numerical and time-refined solution) and uniform time grids ($\Delta t_{\textsc{ref}} = 0.01$ and $\Delta t _{\textsc{num}} = 0.02$ for the reference and numerical grids, respectively). 

At each refinement stage, we compute the reference, space-refined, time-refined and numerical solution using the mesh and timestep schedule at the start of that refinement stage. Next, we use the forward numerical solution and run the adjoint model to compute the sensitivities for each stage and step of the goal function. We again choose the goal function as the integral of the final solution over the entire spatial domain, see \eqref{eq:qoi-numerical-experiments}.

We compute the time residuals using the time-refined solution by plugging it into equation \eqref{eqn:temporal_residuals}. Likewise, the space residuals are obtained by plugging in the space-refined solution into equation \eqref{eqn:space_residuals}. Time and space errors are obtained using the computed sensitivities and residuals, according to the error equations \eqref{eqn:temporal-error-impact} and \eqref{eqn:spatial-error-impact}, respectively. 

At the end of each refinement stage, we inspect the contribution from the cells of the meshgrid to the total spatial error and to the spatial error accumulated over each right-hand side partition of \eqref{eqn:split-ode}. We uniformly refine the cells of the numerical grid that are in the 90th percentile of the total spatial error or total spatial error by partition. The reference grid is refined by uniformly refining the updated numerical grid. Likewise, the time grid of the numerical solution is obtained by dividing in half the intervals that are in the 80th percentile of contribution to the total temporal error by absolute value. The reference time grid is obtained by uniformly refining the updated numerical time grid. The updated space and time grids are used in the subsequent refinement stages. 

Figure \ref{fig:bsvd-numerical-time-refinement} shows how the initially uniform timestep schedule between [0, 4] time units is transformed after each subsequent refinement for the numerical solution. Likewise, Figure \ref{fig:bsvd-mesh-and-solution-plot} shows the progression of the numerical spatial grids after each refinement and the computed solution at $T = 4.0$ on the corresponding grid. 


\begin{figure*}[t!]
	\centering
	\begin{subfigure}[t]{0.48\textwidth}
		\centering
		\includegraphics[width=0.9\textwidth,height=0.65\textwidth]{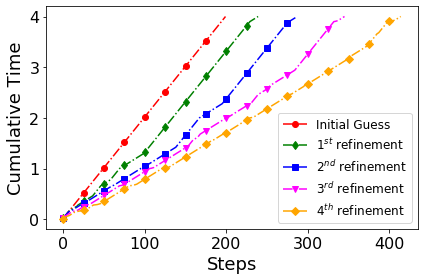}
		\caption{Cumulative Steps}
	\end{subfigure}
	~
	\begin{subfigure}[t]{0.48\textwidth}
		\centering
		\includegraphics[width=\textwidth,height=0.65\textwidth]{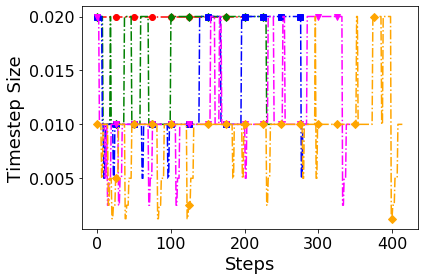}
		\caption{Stepsize}
	\end{subfigure}
	\caption{Progression of timestep schedule and stepsize of each timestep after each refinement of the {\it numerical} solution to the BSVD equation \eqref{eqn:BSVDeqn}.}
	\label{fig:bsvd-numerical-time-refinement}
\end{figure*}

\begin{table}[htb!]
    \footnotesize
    \centering
    \begin{tabularx}{\textwidth}{|c|Y|Y|Y|Y|}
        \hline 
        & Goal function  & Ref. error  & Estimated error  & Est. accuracy \\
        & $\qquad \Psi_\text{ref}$ & $\Psi_\text{ref} - \Psi_\text{num}$ & $\ \ \ \mathcal{E}$ & $(\mathcal{E} - \mathcal{E}_\text{ref})/\mathcal{E}_\text{ref}$ \\
         \hline
          Initial grid & \num{-3.9739E-01} & \num{-5.1521E-01}	 & \num{-5.5565E-01} & \num{-7.8482E-02}  \\ \hline
          $1^{st}$ refinement & \num{-5.0968E-01} & \num{-1.6904E-01} & \num{-1.5573E-01} & 	\num{7.8713E-02	}\\ \hline
          $2^{nd}$ refinement & \num{-5.7183E-01} & \num{-9.5851E-02} & \num{-7.0018E-02} & \num{2.6951E-01	}\\ \hline
          $3^{rd}$ refinement & \num{-6.0371E-01} & \num{-4.3535E-02} & \num{-3.5612E-02} & \num{1.8200E-01	}\\ \hline
          $4^{th}$ refinement & \num{-6.1907E-01} & \num{-1.8293E-02} & \num{-1.8407E-02} & \num{-6.2296E-03}\\
        \hline
    \end{tabularx}
    ~\newline
    \begin{tabularx}{\textwidth}{|c|Y|Y|Y|}
        \hline
        & Time error ($\mathcal{E}_1$)  & Space err. $\!$(diffusion $\mathcal{E}_2$) & Space err. $\!$(reaction $\mathcal{E}_3$) \\
        \hline
        Initial grid  &  \num{-3.2040E-03} & \num{-1.6649E-04} & \num{-5.5228E-01}\\ \hline
        $1^{st}$ refinement & \num{-1.0089E-03} & \num{-6.3589E-03} & \num{-1.4836E-01}\\ \hline
        $2^{nd}$ refinement & \num{-1.2142E-04} & \num{-2.4511E-03} & \num{-6.7446E-02} \\ \hline
        $3^{rd}$ refinement & \num{2.3241E-04} & \num{-3.3892E-03} & 	\num{-3.2455E-02}\\ \hline
        $4^{th}$ refinement & \num{6.2827E-05} & \num{-3.8679E-03} & \num{-1.4602E-02}\\ \hline
    \end{tabularx}
    \caption{Goal function values and error estimates rounded to five digits on successively refined space and time grids for the BSVD equation \eqref{eqn:BSVDeqn}.}
    \label{tab:bsvd_refinement_goals_and_errors}
\end{table}

Table \ref{tab:bsvd_refinement_goals_and_errors} shows the goal function values on the reference solution grid and the reference error as well as the evolution of estimated errors with each refinement of the space and time grid. We observe a first order decrease in the reference and estimated errors with each adaptive refinement. We also observe that the estimated errors initially closely mimic the reference errors when both space and time grid are still relatively coarse. During refinement reference and estimated errors get further apart from each other, before approaching again after the 3rd refinement. This non-uniformity under refinement can be attributed to a pre-asymptotic behavior. The spatially refined grids are shown in Figure~\ref{fig:bsvd-mesh-and-solution-plot} and clearly demonstrate that the protrusion area is well-identified by the error estimation and refinement process.

\FloatBarrier

\begin{figure}[htb!]
    \centering
    \includegraphics[scale=0.4]{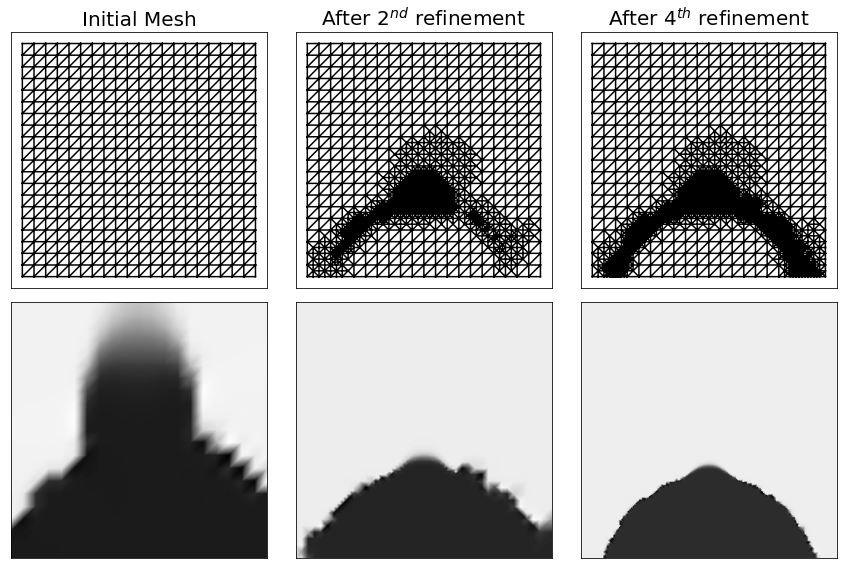}
    \caption{Progression of meshes and solutions for the numerical solution and the BSVD equation \eqref{eqn:BSVDeqn}. Rows: 1) Initial Mesh 2) After $1^{st}$ refinement 3) After $2^{nd}$ refinement 4) After $3^{rd}$ refinement 5) After $4^{th}$ refinement and their corresponding solutions.}
    \label{fig:bsvd-mesh-and-solution-plot}
\end{figure}

%% file: conclusions.tex
\section{Conclusions}
\label{sec:conclusions}

In numerical simulations of multiphysics systems various \\ sources of error can pollute the overall solution. For safety applications and for applications where the decision process crucially depends on the simulation predictions, it is essential to quantify these errors so as to modify the solution parameters, such as adapting the spatial grid, adapting the time grid, and using different finite elements, to obtain a better solution.

In this work, we developed a goal oriented a posteriori error estimation framework using the generalized structure additive Runge-Kutta (GARK) methods for both the forward and adjoint system. We derived the discrete GARK adjoint and proved that it converges at the same rate as the forward GARK method. We highlighted the role of residuals from a time-stepping standpoint and revealed their structure for a number of multiphysics problem scenarios. The method allows us to not only accurately estimate errors, but also to attribute these errors to different physical processes, or different domains of the grid, when the domain partitioning is recast in the GARK setting. A detailed description of the implementation of the framework was given along with the various types of reference solutions needed by the framework. We used the \texttt{Fenics} library to build PDE based numerical examples and implemented our time integrators by extending the \texttt{Gryphon} framework.

Numerical experiments on a reaction-diffusion system showed second order temporal convergence for both the forward method and the adjoint method. These results are in perfect agreement with the theory, established in the paper. Error estimates on two other reaction-diffusion systems were computed and found to accurately predict the reference error. We also performed a space-time adaptive refinement experiment on a reaction-diffusion model, where both the reference error and the estimated error have been found to converge with order one. Guided by our goal-oriented error estimation framework, mesh refinement was carried out locally, in space and time, at those locations, where the solution was difficult to approximate.
